\documentclass[letterpaper]{amsart}
\usepackage{amsthm}
\usepackage{amssymb,latexsym,graphics,enumerate}
\usepackage[mathscr]{eucal}
\usepackage{amsmath,amsfonts,amsthm,amssymb}
\usepackage{color}
\usepackage{hyperref}
\usepackage[left=1in,right=1in,top=1in,bottom=1in]{geometry}
\usepackage[abbrev,lite,nobysame]{amsrefs}

\numberwithin{equation}{section}

\newcommand{\lan}{\langle}
\newcommand{\ran}{\rangle}
\newcommand{\be}{\begin{eqnarray*}}
\newcommand{\bel}{\begin{eqnarray}}
\newcommand{\ee}{\end{eqnarray*}}
\newcommand{\eel}{\end{eqnarray}}
\newcommand{\ba}{\begin{aligned}}
\newcommand{\ea}{\end{aligned}}
\newcommand{\de}{\Delta}

\newcommand{\pa}{\partial}

\newcommand{\pn}{\phi_{\neq}}
\newcommand{\pz}{\lan \phi\ran}

\newtheorem{theorem}{Theorem}[section]
\newtheorem{lemma}[theorem]{Lemma}

\newtheorem{proposition}[theorem]{Proposition}
\newtheorem{corollary}[theorem]{Corollary}
\theoremstyle{definition}
\newtheorem{assumption}{Assumption}[section]
\theoremstyle{definition}
\newtheorem{example}{Example}[section]
\newtheorem{remark}{Remark}[section]

\newcommand{\fn}{f_{\neq}}
\newcommand{\fz}{\lan f\ran}

\newcommand{\norm}[1]{\left\lVert#1\right\rVert}

\newcommand{\set}[1]{\left\{#1\right\}}

\newcommand\N{{\mathbb N}}
\newcommand\R{{\mathbb R}}
\newcommand\T{{\mathbb T}}
\newcommand\Z{{\mathbb Z}}

\newcommand{\cS}{\mathcal{S}}

\newcommand{\ZZ}{\mathbb{Z}}

\newcommand{\TT}{\mathbb{T}}

\newcommand{\rn}{\rho_{\neq}}
\newcommand{\rz}{\langle \rho\rangle}
\newcommand{\cz}{\langle c\rangle}
\newcommand{\cn}{c_{\neq}}

\usepackage{verbatim}
\usepackage{esint}
\usepackage{bm}
\usepackage{enumitem}
\usepackage{amsmath,amsthm}

\usepackage{color}

\def\e{{\rm e}}

\def\dd{{\rm d}}
\def\sign{{\rm sign}}
\def\ddt{{\frac{\dd}{\dd t}}}
\def\R {\mathbb{R}}
\def\ZZ {\mathbb{Z}}
\def\Re{{\rm Re}}
\def\Im{{\rm Im}}
\def \l {\langle}
\def \r {\rangle}
\def\T {{\mathbb T}}
\def\de{{\partial}}

\def\N{{\mathbb{N}}}

\begin{document}

\title[Helical flow and applications]{Dissipation enhancement of planar helical flows and applications to three-dimensional Kuramoto-Sivashinsky and  Keller-Segel equations}

\author[]{Yuanyuan Feng}
\email{yzf58@psu.edu}
\address{Department of Mathematics, Penn State University, University Park, PA
16802, USA}

\author[]{Binbin Shi}
\email{binbinshi@sjtu.edu.cn}
\address{School of Mathematical Sciences, Shanghai Jiao Tong University,
Shanghai,
200240, P. R. China}

\author[]{Weike Wang}
\email{wkwang@sjtu.edu.cn}
\address{School of Mathematical Sciences and Institute of Natural Science, Shanghai Jiao Tong University,
Shanghai,
200240, P. R. China}

\begin{abstract}
We introduce the planar helical flows on three dimensional torus and study the dissipation enhancement of such flows.  We then use such flows as transport flows to solve the three dimensional advective Kuramoto-Sivashinsky and Keller-Segel equations. The global well-posedness of the Kuramoto-Sivashinsky equation is achieved when the  linearized operator does not have growing mode in the direction orthogonal to the flow. The global classical solution of the three dimensional Keller-Segel is ensured for any size of the torus with arbitrarily large initial data.
\end{abstract}



\maketitle

\section{Introduction}
Let $v(x_1,x_2,y)$ be divergence free vector field on $(x_1, x_2, y)\in \T^3=[0,L_1]\times[0, L_2]\times[0, L_3]$. We study the dissipation enhancement of the linear advection diffusion equation:
\begin{align}\label{e:diffusion}
\pa_t \theta +v(y)\cdot\nabla \theta +\nu(-\Delta)^\gamma \theta=0\,,
\end{align}
with $\gamma=1$ or $2$. Here $\nu$ represents the strength of the diffusivity. When $\gamma=1$, it is the normal advection diffusion equation. When $\gamma=2$, it describes the advection hyper-diffusion equation.
From a standard energy estimate, the quantity $\norm{\theta}_{L^2}$ will decay at a rate of order $O(\nu^{-1})$. While when $v$ is relaxation enhancing or mixing at some rate, the  $L^2$ norm of the energy will dissipate much faster than $O(\nu^{-1})$. For details, one can read~\cite{Constantin.2008,FengIyer2019dissipation, ZelatiDelgadinoElgindi2020relation}. When $v$ is a shear flow, it will enhance the energy dissipation for the components lying outside the null space of the shear and people have done many works on this~\cite{ZelatiDelgadinoElgindi2020relation,BedrossianZelati2017enhanced,
Wei2019diffusion}.

In this paper, we first introduce the planar helical flow $v$ on $\T^3$, which is defined as
\begin{align}\label{e:helical}
v(y)=(u(y)\sin( 2\pi y /L_3), u(y)\cos (2\pi y/L_3),0)\,,
\end{align}
where $u$ is some smooth periodic function with period $L_3$. People have studied helical flows with diffrerent definitions. For instance in the work~\cite{Lopes2014planar,Ettinger2009global}, they defined helical flows based on some helical symmetry.  In our definition,  we call it `` planar helical" because the trajectory of the flow  wraps around the torus.

For $(x_1, x_2, y)\in \T^3$, we then study the dissipation enhancement of the  advection diffusion equation~\eqref{e:diffusion}  advected by the flow $v$ given in~\eqref{e:helical}.  We will show the above equation will satisfy an enhanced dissipation in the sense of $L^2$ norm for any initial data lying outside of the nullspace of the flow. More precisely, we define the linear operator $H_\nu$ by
\begin{align}\label{e:Hnu}
H_\nu:=\nu(-\Delta)^\gamma+u(y) \sin (2\pi y/L_3) \,\pa_{x_1}  +u(y)\cos(2\pi y/L_3) \,\pa_{x_2}\,,
\end{align}
with $\gamma=1$ or $2$. We will prove under some assumption on $u$,  for the components of the solution orthogonal to the kernel of the transport operator, the corresponding semigroup $e^{-tH\nu}$ will decay with a rate of order $\lambda_\nu$  in the sense of the $L^2$ norm, where $\nu/\lambda_\nu\to 0$ as $\nu\to 0$.  We will also give two examples  of $v$, i.e., $v=(\sin(2\pi y/L_3), \cos(2\pi  y/L_3), 0)$ and $v=(\cos(2\pi y/L_3)\sin(2\pi y/L_3), (\cos (2\pi y/L_3))^2, 0)$, and will check for both choices of $v$ , $\lambda_\nu$ is of order $\nu^{1/(1+\gamma)}$.

We then apply such planar helical flows to the nonlinear equations on $\T^3$, Kuramoto-Sivashinsky and Keller Segel equations. 
Recall that the advective Kuramoto-Sivashinsky equation is given by
\begin{align}\label{e:KSA}
\de_t \phi+Av\cdot \nabla \phi+\frac{1}{2}|\nabla \phi|^2+\Delta^2\phi+\Delta\phi=0\,,
\end{align}
where $A$ is a parameter representing the amplitude of the flow $v$. By rescaling time, it can be equivalently written as
\begin{align}\label{eq:AKSE}
\de_t \phi+v\cdot \nabla \phi+\frac{\nu}{2}|\nabla \phi|^2+\nu\Delta^2\phi+\nu\Delta\phi=0\,,
\end{align}
with $\nu=A^{-1}$. The classical Kuramoto-Sivashinsky equation (by choosing $A=0$ in~\eqref{e:KSA}) models the propagation of the flame front. The analysis of the Kuramoto-Sivashinsky equations in one space dimension is well developed. In one dimension a priori norm estimates on the solution are closed which lead to a good control on the $L^2$ norm of the solution \cite{Bronski2006uncertainty,Collet1993analyticity,Otto2009optimal,
Goldman2015new,Giacomelli2005new,Goluskin2019bounds}. In the dimension greater than one, much less progress has been made since it is a fourth order equation and lack of maximum principle. The global well-posedness in dimesion $d=2$ is only known under some restrictive assumptions, such as for thin domains and for the anisotropically reduced Kuramoto-Sivashinsky equation~\cite{Benachour2014anisotropic,Larios2020well,Sell1991local}, without growing modes~\cite{Ambrose2019global,FengmMazzucato2020global}, or with only one growing mode in each direction~\cite{Ambrose2021global}, for small data.

In~\cite{FengmMazzucato2020global}, one of the authors studied the Kuramoto-Sivashinsky equation with a transport flow, and proved that global existence of the advective Kuramoto-Sivashinsky equation can be achieved for arbitrary data provided
the transport flow $v$ is relaxation enhancing with sufficiently small dissipation time. In this case, the effect generated by the growing modes can be efficiently damped due to enhanced dissipation. In~\cite{ZelatiDolceFengMazzucato2021global}, one of the authors studied the global existence of the advective Kuramoto-Sivashinsky equation with a steady shear. For the case of shear flow, the transport operator has a large kernel space, which means it does not generate any dissipation enhancement on the kernel component. However, if we project the solution to the kernel space of the shear, it reduces to a modified one dimensional Kuramoto-Sivashinsky equation.  For the component orthogonal to the kernel space, by choosing the flow with a large amplitude (or with $\nu$ small), the property of dissipation enhancement will ensure the global existence provided the equation of the coupled component in the kernel space does not have any growing mode, i.e., the size of the projected domain is less than $2\pi$.  

In dimension $d=3$, an added shear flow does not work since the kernel space of the shear flow is two dimensional and the global existence of the modified two dimensional Kuramoto-Sivashinsky equation is not known how to achieve. Instead, we use the planar helical flows as defined in~\eqref{e:helical}.  So that the  component in the kernel space of the flow can be reduced to a modified one dimensional Kuramoto-Sivashinsky equation. For the component orthogonal to the kernel, due to the property of dissipation enhancement  of the flow, we will prove that as long as the domain in the direction orthogonal to the flow does not have any growing mode, i.e., $L_3<2\pi$,  there exists $\nu=\nu(\phi_0)$ small  such that the $L^2$ norm of the component orthogonal to the kernel will stay bounded.  The global well-posedness of the whole solution is thus established. Later we will show similar idea also works for the three dimensional Keller-Segel equation. 

We first state the global existence of the three dimensional advective Kuramoto-Sivashinsky equation as below.

\begin{theorem}\label{t:KS}
Let the domain $\T^3=[0,L_1]\times[0, L_2]\times[0, L_3]$ satisfy $L_3< 2\pi$. Let $\phi_0\in H^1(\mathbb{T}^3)$, and let  $u:[0,L_3)\to \R$ be a smooth function satisfying Assumption~\eqref{a:lowerbddE}. Then  there exists $0<\nu_0<1$ depending on $L_1,\,L_2, L_3, u$ and $\phi_0$ with the following property: for any $0<\nu<\nu_0$, there exists a global-in-time weak solution $\phi$ of \eqref{eq:AKSE} with initial data $\phi_0$ such that $\phi\in L^\infty([0,T),L^2)\cap L^2([0,T),H^2)$ for all $0<T<\infty$.
\end{theorem} 

We make a few remarks in below.
\begin{remark}\textrm{\\}
\begin{itemize}
\item[(1)]The proof of Theorem~\eqref{t:KS} is based on a bootstrap argument, which is similar to the previous work of one of the authors~\cite{ZelatiDolceFengMazzucato2021global}. We will find later the choice of $\nu_0$ based on the estimate of $\lambda_\nu$. Hence achieving better estimate of the dissipation rate $\lambda_\nu$ will improve the threshold on global existence.
\item[(2)] When the dimension is $d\geq 4$ in general, one may also use similar idea to add a transport flow which is relaxation enhancing outside of a one dimensional kernel space. In higher dimensions the detailed estimates may get complicated and some embedding properties may fail. Besides that, we point out one more difficulty lies in how to construct such flows. The construction of the planar helical flow given in~\eqref{e:helical} can not be immediately generalized to higher dimensions. 
\end{itemize}

\end{remark}
We treat the three dimensional Keller-Segel equation with similar idea. By adding the advective flow with some amplitude and then rescaling time, finally the equation becomes
\begin{equation}\label{e:KSegel}
\begin{cases}
\partial_t\rho+v\cdot \nabla \rho-\nu\Delta\rho+\nu\nabla\cdot(\rho \nabla c)=0\,,\\
-\Delta c=\rho-\overline{\rho}\,,  \\
\rho(0,x_1,x_2,y)=\rho_0(x_1,x_2,y)\,.
\end{cases}
\end{equation}

As we know, the solution to the classical Keller-Segel system may  blow up in finite time when the dimension is  larger than one. More precisely, in the case the dimension $d=2$, if $L^1$ norm of initial data $\rho_0$ is less than $8\pi$, then there exists a unique global solution, and if $L^1$ norm of initial data $\rho_0$ exceeds $8\pi$, the solution blows up in finite time. For dimension $d\geq 3$, the blow-up occurs for solutions with arbitrary small  $L^1$ norm of the initial data. For more details, one can read~ \cite{Masmoudi.2008,Blanchet.2006,Corrias.2004,Luckhaus.1992,Nagai.1995,Winkler.2019}. An interesting question araises whether one can suppress the finite time blow-up by the stabilizing effect of the moving fluid.  Kiselev and Xu \cite{Kiselev.2016} considered the relaxation enhancing flow which was introduced in \cite{Constantin.2008}. They proved the solution of the advective Keller-Segel equation does not blow-up in finite time provided the amplitude of the relaxation enhancing flow is large enough. Later for the generalized  Keller-Segel equation with fractional Laplacian and relaxation enhancing flow, the global ewell-posedness are disscussed in~\cite{Hopf.2018,Shi.2019,Shi.2021}. Bedrossian and He~\cite{Bedrossian.2017} proved  that shear flows can also suppress the blow-up of the solution.  More precisely, they proved that in the two dimensional case the solution is global in time. While in the three dimensional case, the global well-posedness is guaranteed only when the initial mass is less than $8\pi$. He and Tadmor \cite{Tadmor.2019} investigated the effect of a flow of the ambient environment introduced by harmonic potentials in the two dimensional case. They showed that the enhanced ambient flow doubles the amount of allowable mass which evolve to the global smooth solutions. For the high dimensional case,  similar questions have been studied by He, Tadmor and Zlato\v{s} \cite{Tadmor.2021} recently. Zeng, Zhang and Zi \cite{Zeng.2021} considered the two dimensional Keller-Segel-Navier-Stokes system near the Coutte flow. They showed that if flow is large enough,  the solution to the equation was global existence. For the parabolic-parabolic case, some results can be referred to \cite{He.2018,Zeng.2021}.

In this paper, we will prove the global existence of the solution of the three dimensional Keller-Segel equation with an advective planar helical flow. This question is motivated by works of Wang et. al. \cite{Chen.2016}  and Bedrossian et. al \cite{Bedrossian.2017}.  Similiar to the three dimensional Kuramoto-Sivashinsky equation, the kernel component of the flow satisfies a modified one dimension Keller-Segel equation and the othogoanl component will be efficiently effected by the flow.  We will prove that by choosing the flow with large amplitude (or equivalently, with $\nu$ small in~\eqref{e:KSegel}), the $L^2$ norm of the solution is bounded uniformly in time.  Actually, since the $L^2$ estimate is supercritical for the three dimensional Keller-Segel equation (see~\cite{Kiselev.2016}), the global classical solution can be achieved.  To point out, the shear flow is used as the transport flow and the global existence of the advective Keller-Segel equation in dimension $d=3$ is achieved with the restriction of mass less than $8\pi$ in~\cite{Bedrossian.2017}.  In this paper, we prove the global existence with any initial data, instead with the $8\pi$ restriction.

\begin{theorem}\label{t:KSegel}
Let $\rho_0\in H^1(\T^3)\cap L^\infty(\mathbb{T}^3)$, and let  $u:[0,L_3)\to \R$ be a smooth function satisfying Assumption~\eqref{a:lowerbddE}. Then  there exists $0<\nu_0<1$ depending on $L_1,\,L_2, L_3, u$ and $\rho_0$ with the following property: for any $0<\nu<\nu_0$, there exists a global-in-time classical solution $\rho$ of \eqref{e:KSegel} with initial data $\rho_0$.
\end{theorem}

Throughout the paper, we use standard notations to denote function spaces and use $C$ to denote a generic constant which may vary from line to line.

The paper is organized as follows. In Section~\ref{s:helical}, we give the estimate on the rate of dissipation enhancement of the introduced planar helical flows and give two examples of such flows. In Section~\ref{s:KS}, we prove Theorem~\ref{t:KS} to show the global existence of the three dimensional  Kuramoto-Sivashinsky equation. In Section~\ref{s:KSegel}, we prove Theorem~\ref{t:KSegel} to establish the global well-posedness of the three dimensional  Keller-Segel equation.

\section{Dissipation enhancement of the planar helical flow}\label{s:helical}
Let $(X,\norm{\cdot})$ be a complex Hilbert space and let $H$ be a closed, densely defined operator on $X$. $H$ is m-accretive if the left open half-plane is contained in the resolvent set with
\begin{align*}
(H+\lambda)^{-1}\in \mathcal{B}(X)\,,\quad \norm{(H+\lambda)^{-1}}\leq (\Re \lambda)^{-1}\,,\quad \text{for } \Re\lambda >0\,,
\end{align*}
where $\mathcal{B}(X)$ denotes the set of bounded linear operators on $X$. As shown in~\cite{Wei2019diffusion},  using a Gearhart-Pr\"uss type theorem with a sharp bound for m-accretive operators, the decay property of the semigroup $e^{-tH}$ can be bounded by 
\begin{align}\label{e:maccretive}
\norm{e^{-tH}} \leq e^{-t\Psi(H)+\pi/2}\,,\quad\forall t\geq 0\,,
\end{align}
where $\Psi(H)$ is defined by 
\begin{align}
\Psi(H)=\inf\set{\norm{(H-i\lambda)g}:\, g\in D(H), \,\lambda\in \R,\, \norm{g}=1}\,.
\end{align}

Let $k=(k_1, k_2)\neq 0$ with $\Big(\frac{L_1k_1}{2\pi}, \frac{L_2k_2}{2\pi}\Big)\in \Z^2$. We consider the operator localized to the $k$th Fourier mode:
\begin{align}
H_{\nu, k}:=\nu(-\Delta_k)^\gamma+ik_1u(y) \sin (2\pi y/L_3)  +ik_2 u(y)\cos(2\pi y/L_3) \,,\quad \Delta_k=-k_1^2-k_2^2+\de_{yy}\,,
\end{align}
with $\gamma=1$ or $2$. Similar to the argument in~\cite{Wei2019diffusion}, it can be verified that $H_{\nu, k}$ is m-accretive on $L^2(\T^1)$ with domain $H^{2\gamma}(\T^1)$. As a consequence, the estimate~\eqref{e:maccretive} can be applied and we only need to find a lower bound of $\Psi(H_{\nu,k})$.

To establish a lower bound of $\Psi(H_{\nu,k})$, we give the following assumption on $u$, which is inspired by  the previous work of one of the authors in~\cite{ZelatiDolceFengMazzucato2021global}.
\begin{assumption}\label{a:lowerbddE}
There exist $m, N\in \N$, $c_1>0$ and  $\delta_0\in(0,L_3)$ with the property that, for
  any $\lambda , \alpha \in \R$ and any $\delta\in(0,\delta_0)$, there exist $n\leq N$ and points $y_1,\ldots y_n\in [0,L_3)$
such that
\begin{align}\label{eq:lowerbddE2}
|u(y)\sin(2\pi y/L_3+\alpha)-\lambda|\geq c_1 \left(\frac{\delta}{L_3}\right)^m, \qquad \forall \  |y-y_j|\geq \delta, \quad \forall j\in \{1,\ldots n\}.
\end{align}
\end{assumption}
Based on such an assumption on $u$, some calculations are carried out and the lower bounds of $\Psi(H_{\nu,k})$ are achieved in the next Proposition.

\begin{proposition}\label{p:derate}
Let $u$ satisfy Assumption~\eqref{a:lowerbddE}. Let $k\neq(0, 0)$ and $\nu|k|^{-1}\leq 1$. There exists a constant $\epsilon_0>0$, independent of $\nu$ and $k$, such that
\begin{itemize}
\item[(i)] In the case $\gamma=1$, we have $\Psi(H_{\nu, k})\geq \epsilon_0 \nu^{\frac{m}{m+2}}|k|^{\frac{2}{m+2}}$.
\item[(ii)] In the case $\gamma=2$, we have $\Psi(H_{\nu, k})\geq \epsilon_0 \nu^{\frac{m}{m+4}}|k|^{\frac{4}{m+4}}$.
\end{itemize}
\end{proposition}
Before proving the proposition, we state a direct corollary as follows.
\begin{corollary}\label{c:decay}
In the hypotheses of Proposition~\eqref{p:derate}. Let $P_k$ denote the $ L^2$ projection onto the $k$th Fourier mode. Then for every $t\geq 0$, it holds that
\begin{align}\label{e:derate}
\norm{e^{-tH_{\nu}}P_k}\leq e^{-\lambda_\nu t+\pi/2}\,,
\end{align}
where $\lambda_\nu=\epsilon_0 \nu^{\frac{m}{m+2}}|k|^{\frac{2}{m+2}}$ for $\gamma=1$ and $\lambda_\nu=\epsilon_0 \nu^{\frac{m}{m+4}}|k|^{\frac{4}{m+4}}$ for $\gamma=2$\,.
\end{corollary}
\begin{proof}[Proof of Proposition~\eqref{p:derate}]
We first prove the case when $\gamma=1$. For notational conciseness, we denote the $L^2$ norm as $\norm{\cdot}$ and the Hermitian inner product in $L^2$ as $\lan \cdot, \cdot\ran$.  For any fixed $\lambda\in \R$ and $g\in D(H_{\nu,k})$ with $\norm{g}=1$, we denote
\begin{align*}
H&:=-\nu\Delta_k+i k_1 u(y)\sin(2\pi y/L_3) +i k_2 u(y)\cos(2\pi y/L_3)-i\lambda\\
&=-\nu \Delta_k+ i |k| \big( u(y)\sin(2\pi y/L_3+\alpha_k)-\tilde \lambda\big)\,, \quad \tilde \lambda := \frac{\lambda}{|k|}\,.
\end{align*}
Denote
\begin{align}
E:=\{y\in [0, L_3): |y-y_j|\geq \delta, \quad \forall j\in\{1,\ldots, n\} \}\,,
\end{align}
where $\{y_j\}_{1\leq j\leq n}$ are the points in Assumption~\eqref{a:lowerbddE}.
Let $\chi:[0, L_3)\to [-1, 1]$ be a smooth approximation of $\sign (u(y)\sin (2\pi y/L_3+\alpha_k)-\tilde\lambda)$ satisfying $\norm{\chi'}_{L^\infty} \leq c_2 \delta^{-1}$, $\norm{\chi''}_{L^\infty}\leq c_2\delta^{-2}$, $\chi(y)(u(y)\sin (2\pi y/L_3+\alpha_k)-\tilde\lambda)\geq 0$ and
\begin{align}
\chi(y)(u(y)\sin (2\pi y/L_3+\alpha_k)-\tilde\lambda)=|u(y)\sin (2\pi y/L_3+\alpha_k)-\tilde\lambda|\,,\quad \forall y\in E\,.
\end{align}
Such a function $\chi$ can be constructed via a standard mollification argument. We observe that
\begin{align}
\lan Hg,\chi g\ran &=-\nu \lan \Delta_k g, \chi g\ran+i|k|\lan \big(u(y)\sin (2\pi y/L_3+\alpha_k)-\tilde \lambda\big)g, \chi g\ran\\
&=\nu\lan\de_y g, \chi' g\ran +\nu\lan \de_y g, \chi \de_y g\ran+\nu|k|^2\lan g, \chi g\ran+i|k|\lan  \big(u(y)\sin (2\pi y/L_3+\alpha_k)-\tilde \lambda\big)g, \chi g\ran\,,
\end{align}
which implies
\begin{align}
\Im \lan Hg, \chi g\ran=\nu\Im \lan \de_y g, \chi' g\ran+  |k|\lan  \big(u(y)\sin (2 \pi y/L_3+\alpha_k)-\tilde \lambda\big)g, \chi g\ran\,.
\end{align}
Using the properties of $\chi$, we get
\begin{align}\label{e:lowerbddelta1}
|k|\lan  \big(u(y)\sin (2\pi y/L_3+\alpha_k)-\tilde \lambda\big)g, \chi g\ran \leq \norm{Hg}\norm{ g}+\frac{c_2\nu}{\delta}\norm{\de_y g}\norm{g}\,.
\end{align}
By~\eqref{eq:lowerbddE2}, we have
\begin{align}\label{e:lowerbddelta2}
\lan  \big(u(y)\sin (2\pi y/L_3+\alpha_k)-\tilde \lambda\big)g, \chi g\ran \geq \int_E |u(y)\sin(2 \pi y/L_3+\alpha_k)-\tilde \lambda||g(y)|^2\,dy\geq c_1\Big(\frac{\delta}{L_3}\Big)^m \int_E |g(y)|^2\,dy\,.
\end{align}
We note that $\norm{\de_y g}^2\leq 1/\nu \norm{Hg}\norm{g}$ since $\Re\lan Hg, g\ran=\nu \norm{\de_y g}^2+|k|^2\norm{g}^2$. Using this in~\eqref{e:lowerbddelta1} and ultilizing~\eqref{e:lowerbddelta2}, one has
\begin{align}\label{eq:estimateE1}
\nonumber
\int_E|g(y)|^2\,dy&\leq \frac{1}{c_1|k|}\Big(\frac{L_3}{\delta}\Big)^m\Big(\norm{Hg}\norm{g}+\frac{c_2\nu^{1/2}}{\delta}\norm{Hg}^{1/2}\norm{g}^{3/2}\Big)\\
&\leq \Big(\frac{1}{c_1|k|}\Big(\frac{L_3}{\delta}\Big)^m+\frac{\tilde c_2 \nu}{L_3^2|k|^2}\Big(\frac{L_3}{\delta}\Big)^{2m+2}\Big)\norm{Hg}\norm{g}+\frac{1}{4}\norm{g}^2\,.
\end{align}
On the other hand, since $E^c$ is of size less than $N\delta$, it holds that
\begin{align}\label{eq:estimateEc1}
\int_{E^c} |g(y)|^2\dd y\leq N\delta \|g\|^2_{L^\infty}&\leq CN\delta \left(\|g\|\|\de_yg\| +\| g\|^2\right)\notag\\
&\leq CN\delta \left(\frac{1}{\nu^{1/2}}\|g\|^{3/2}\| Hg\|^{1/2} + \| g\|^2\right)\notag\\
&\leq \frac{C ( N\delta)^2}{\nu}\| H g\|\norm{g} +\frac{1}{2}\| g\|^2\,.
\end{align}
where  $\delta_0$ is taken small enough so that $\delta\leq 1/(4CN)$ for any $\delta\in (0, \delta_0)$. Adding up~\eqref{eq:estimateE1} and~\eqref{eq:estimateEc1} yields that
\begin{align}
\norm{g}^2\leq 4\Big(\frac{1}{c_1|k|}\Big(\frac{L_3}{\delta}\Big)^m+\frac{\tilde c_2 \nu}{L_3^2|k|^2}\Big(\frac{L_3}{\delta}\Big)^{2m+2}+\frac{C(N\delta)^2}{\nu}\Big)\norm{Hg}\norm{g}\,.
\end{align}
By taking
\begin{align}
\frac{\delta}{L_3}=c_3\Big(\frac{\nu}{|k|}\Big)^{\frac{1}{m+2}}\,,
\end{align}
with $c_3$ small enough, we obtain
\begin{align}
\norm{Hg}\geq \epsilon_0\nu^{\frac{m}{m+2}}|k|^{\frac{2}{m+2}}\norm{g}=\epsilon_0\nu^{\frac{m}{m+2}}|k|^{\frac{2}{m+2}}\,.
\end{align}
Here $\epsilon_0$ is some constant independent of $\nu$ and $k$. Finally we get
\begin{align}
\Psi(H_{\nu,k})\geq \epsilon_0\nu^{\frac{m}{m+2}}|k|^{\frac{2}{m+2}}\,.
\end{align}
When $\gamma=2$, the proof is similar. Instead we define
\begin{align*}
H&:=\nu\Delta_k^2+i k_1 u(y)\sin(2\pi y/L_3) +i k_2 u(y)\cos(2\pi y/L_3)-i\lambda\\
&=\nu \Delta_k^2+ i |k| \big( u(y)\sin(2\pi y/L_3+\alpha_k)-\tilde \lambda\big)\,, \quad \tilde \lambda := \frac{\lambda}{|k|}\,,
\end{align*}
and observe that
\begin{align}
\lan Hg,\chi g\ran &=\nu \lan \Delta^2_k g, \chi g\ran+i|k|\lan \big(u(y)\sin (2\pi y/L_3+\alpha_k)-\tilde \lambda\big)g, \chi g\ran\\
&=\nu\lan \Delta_k g, \chi'' g\ran+2\nu \lan \Delta_k g, \chi' \de_y g\ran+\nu\lan \Delta_k g, \chi \Delta_k g\ran+i|k|\lan  \big(u(y)\sin (2\pi y/L_3+\alpha_k)-\tilde \lambda\big)g, \chi g\ran\,.
\end{align}
This implies
\begin{align}
\Im \lan Hg, \chi g\ran=\nu\Im \lan \Delta_k g, \chi'' g\ran+ 2\nu \Im\lan \Delta_k g, \chi'\de_y g\ran+ |k|\lan  \big(u(y)\sin (2\pi y/L_3+\alpha_k)-\tilde \lambda\big)g, \chi g\ran\,.
\end{align}
Using the properties of $\chi$  and the interpolation inequality $\norm{\de_y g}^2\leq \norm{\Delta_k g}\norm{g}$, we further get
\begin{align}\label{e:lowerbdtmp1}
|k|\lan  \big(u(y)\sin (2 \pi y/L_3+\alpha_k)-\tilde \lambda\big)g, \chi g\ran \leq \norm{Hg}\norm{g}+\frac{c_2\nu}{\delta^2}\norm{\Delta_k g}\norm{g}+\frac{c_2\nu}{\delta}\norm{\Delta_k g}^{3/2}\norm{g}^{1/2}\,.
\end{align}
We note that $\norm{\Delta_k g}^2\leq 1/\nu \norm{Hg}\norm{g}$ since $\Re\lan Hg, g\ran=\nu \norm{\Delta_k g}^2$.
Combing this with~\eqref{e:lowerbdtmp1} and~\eqref{e:lowerbddelta2}, we get
\begin{align}\label{eq:estimateE}
&\int_E |g(y)|^2\dd y\leq \frac{1}{c_1 |k|} \left(\frac{L_3}{\delta}\right)^{m}\left[\| Hg\| \|g\|+\frac{c_2\nu}{\delta^2} \|\Delta_k  g \|\|g\|+\frac{c_2\nu}{\delta}  \|\Delta_k  g \|^{3/2}\|g\|^{1/2}\right]\notag\\
&\qquad\leq \frac{1}{c_1 |k|} \left(\frac{L_3}{\delta}\right)^{m}\| Hg\| \|g\|
+\tilde{c}_2\left(\left(\frac{\nu}{|k|\delta^2}\right)^2\left(\frac{L_3}{\delta}\right)^{2m} +\left(\frac{\nu}{|k|\delta}\right)^{4/3}\left(\frac{L_3}{\delta}\right)^{\frac{4m}{3}} \right)\|\Delta_k  g \|^2 +\frac14\|g\|^2\notag\\
&\qquad\leq \left(\frac{1}{c_1 |k|}\left(\frac{L_3}{\delta}\right)^{m}+\frac{\tilde{c}_2}{\nu}\left(\left(\frac{\nu}{|k|L_3^2}\right)^2\left(\frac{L_3}{\delta}\right)^{2m+4} +\left(\frac{\nu}{|k|L_3}\right)^{4/3}\left(\frac{L_3}{\delta}\right)^{\frac{4}{3}(m+1)} \right)\right)\| Hg\| \|g\|+\frac14\|g\|^2.
\end{align}
On the other hand,  we have
\begin{align}\label{eq:estimateEc}
\int_{E^c} |g(y)|^2\dd y\leq N\delta \|g\|^2_{L^\infty}&\leq CN\delta \left(\|g\|\|\de_yg\| +\| g\|^2\right)\notag\\
&\leq CN\delta \left(\|g\|^{3/2}\| \Delta_k g\|^{1/2} + \| g\|^2\right)\notag\\
&\leq C ( N\delta)^4\| \Delta_k g\|^2 +\frac{1}{2}\| g\|^2\notag\\
&\leq  \frac{C( N\delta)^4}{\nu}\|Hg\|\|g\|  +\frac{1}{2} \| g\|^2,
\end{align}
where we assume $\delta_0$ is small enough so that $\delta\leq 1/(4CN)$ for any $\delta\in (0, \delta_0)$. Adding up~\eqref{eq:estimateE} and~\eqref{eq:estimateEc} yields that
\begin{align}
\|g\|\leq 4\left(\frac{1}{c_1 |k|}\left(\frac{L_3}{\delta}\right)^{m}+\tilde{c}_2\frac{\nu}{(|k| L^3_2)^2}\left(\frac{L_3}{\delta}\right)^{2m+4}+\tilde{c}_2\frac{\nu^{1/3}}{\left(|k|L_3\right)^{4/3}}\left(\frac{L_3}{\delta}\right)^{\frac{4}{3}(m+1)} +\frac{C(N\delta)^4}{\nu}\right)\| Hg\|.
\end{align}
We then take
\begin{align}
\frac{\delta}{L_3}=c_3\Big(\frac{\nu}{|k|}\Big)^{\frac{1}{m+4}}\,,
\end{align}
with $c_3>0$ small enough, which then gives
\begin{align}
\norm{Hg}\geq \epsilon_0\nu^{\frac{m}{m+4}}|k|^{\frac{4}{m+4}}\norm{g}=\epsilon_0\nu^{\frac{m}{m+4}}|k|^{\frac{4}{m+4}}\,,
\end{align}
for some constant $\epsilon_0$ independent of $\nu$ and $k$. Hence we get
\begin{align}
\Psi(H_{\nu,k})\geq \epsilon_0\nu^{\frac{m}{m+4}}|k|^{\frac{4}{m+4}}\,.
\end{align}
as desired.

\end{proof}

\begin{example}\label{exam:2.1}
We give two examples of the choice of $u$ on $\T^1=[0, L_3]$.
\begin{itemize}
\item Take $u(y)=1$. Then Assumption~\eqref{a:lowerbddE} holds with $m=2$.
\item Take $u(y)=\cos (2\pi y/L_3)$. Then Assumption~\eqref{a:lowerbddE} holds with $m=2$.
\end{itemize}
\end{example}
The detailed verification heavily depends on the explanation of Example 3.1 in~\cite{ZelatiDolceFengMazzucato2021global}. There it was verified that:

\textit{There exist $m$, $N\in\N$, $c_1>0$ and $\delta_0 \in (0, 2\pi)$ with the property that for any $\lambda\in \R$ and any $\delta\in (0, \delta_0)$, there exist $n\leq N$ and points $y_1,\cdots,y_n\in [0, 2\pi) $ such that
\begin{align}\label{e:}
|(\sin y)^m-\lambda|\geq c_1\delta^{\max{(m,2)}}\,,\quad \forall |y-y_j|\geq \delta, \forall j\in\{ 1,\cdots, n\}\,.
\end{align}
}
Here in our case, when $u=1$,  the estimate~\eqref{eq:lowerbddE2} follows directly by changing variables. When $u=\cos(2\pi y/L_3)$, we note that $|\cos(2\pi y/L_3)\sin(2\pi y/L_3+\alpha) -\lambda|=\frac{1}{2}|\sin(4\pi y/L_3+\alpha)+\sin \alpha -2\lambda|$. Then the result follows by using a standard rescaling argument.

\section{Global existence of the 3D Kuramoto-Sivashinsky equation with planar helical flow}\label{s:KS}
In this section, we deal with the global existence of the three dimensional advective Kuramoto-Sivashinsky equation. As proved in~\cite{FengmMazzucato2020global}, by showing the mild solution is a weak solution, the well-posedness of the solution is ensured once the $L^2$ norm of the solution is bounded. Hence to prove the global existence, we only need to prove boundedness of the $L^2$ norm. The idea of the proof is similar to the two dimensional case with shear flows in~\cite{ZelatiDolceFengMazzucato2021global}. The main difference lies in the Sobolev inequalities we use in the argument since the dimension has changed. We write the full proof here for the completeness.

Given $f\in L^2(\T^3)$, we denote
\begin{equation}\label{e:decompose}
\fz(y)=\frac{1}{L_1L_2}\int_{\T^2} f(x_1,x_2,y) \dd x_1\dd x_2, \qquad
 \fn(x_1,x_2,y)=f(x_1,x_2,y)-\fz(y).
\end{equation}
We observe that $\lan f\ran$ corresponds to the projection of  $f$ onto the kernel of the transport operator $u(y)\sin y \,\pa_{x_1}+u(y)\cos y \,\pa_{x_2}$, while $ f_{\neq}$ corresponds to the projection onto the orthogonal complement in $L^2$. Let $\phi$ be the solution of~\eqref{eq:AKSE}. We then derive the equation for $\pz$ and $\pn$. The equation for $\pz$ reads as
\begin{align}\label{e:pz}
\partial_t\pz+\frac{\nu}{2L_1L_2}\int_{\T^2}|\nabla \pn+\nabla\pz|^2\,\dd x_1\,\dd{x_2}+\nu\partial_{y}^4\pz+\nu\partial_{y}^2\pz=0\,,
\end{align}
while $\pn$ satisfies
\begin{align}\label{e:pn}
\partial_t\pn+u(y)\sin y\,\partial_{x_1}\pn+u(y)\cos y\,\partial_{x_2}\pn+\nu\Delta^2\pn
=-\frac{\nu}{2}|\nabla \pn|^2+\frac{\nu}{2}\lan |\nabla\pn|^2\ran-\nu\partial_y\pn\partial_y\pz-\nu\Delta\pn\,.
\end{align}
In the above equation, the kernel component interacts with the projected one only through the term $\de_y \l \phi\r$. For notational ease, we denote $\psi=\partial_y\pz$ and get
\begin{align}\label{e:psi}
\partial_t\psi+\frac{\nu}{2L_1L_2}\int_{\T^2}\partial_y|\nabla \pn|^2\,\dd x_1\dd x_2+\nu\psi\partial_y\psi+\nu\partial_y^4\psi+\nu\partial_y^2\psi=0\,.
\end{align}

Denote $\cS_t=\e^{-t H_\nu}$. Then by Duhamel's formula, we have
\begin{align} \label{eq:pnMild}
\pn(s+t)&=\cS_t(\pn(s))+\int_s^{s+t}\cS_{t+s-\tau}\Big(-\frac{\nu}{2}|\nabla\pn(\tau)|^2
+\frac{\nu}{2}\lan |\nabla\pn(\tau)|^2\ran -\nu\psi(\tau)\partial_y\pn(\tau)-\nu\Delta\pn(\tau)\Big)\,\dd \tau.
\end{align}
Using~\eqref{e:derate}, it implies that
\begin{align}\label{e:duhamel1}
\norm{\pn(t+s)}_{L^2}&\leq \norm{\cS_t(\pn(s))}_{L^2}+C\nu\int_s^{s+t}\big(\norm{\nabla \pn}_{L^4}^2+\norm{\psi}_{L^4_y}\norm{\nabla\pn}_{L^4}+\norm{\Delta\pn}_{L^2}\big)\,\dd \tau\,.
\end{align}
We recall the following Gagliardo-Nirenberg interpolation inequalities on $\T^3$ and $\T^1$ respectively:
\begin{align}\label{e:soblev1}
\norm{\nabla \pn}_{L^4}\leq C\norm{\pn}_{L^2}^{1/8}\norm{\Delta \pn}^{7/8}_{L^2}\,,\quad \norm{\psi}_{L^4_y}\leq C\norm{\psi}_{L^2_y}^{7/8}\norm{\partial_y^2\psi}_{L^2_y}^{1/8}\,.
\end{align}
Then it follows from~\eqref{e:duhamel1} that
\begin{align}\label{e:duhamel}
\nonumber
\norm{\pn(t+s)}_{L^2}&\leq \norm{\cS_t(\pn(s))}_{L^2}+C\nu\int_s^{s+t}\big(\norm{\pn}_{L^2}^{1/4}\norm{\Delta\pn}_{L^2}^{7/4}+\norm{\Delta\pn}_{L^2}\\
\nonumber
&\qquad\qquad\qquad\qquad\qquad\qquad+\norm{\pn}_{L^2}^{1/8}\norm{\Delta\pn}_{L^2}^{7/8}\norm{\psi}_{L^2_y}^{7/8}\norm{\pa^2_y\psi}_{L^2_y}^{1/8}\big)\,\dd \tau\\
&\leq \norm{\cS_t(\pn(s))}_{L^2}+C\nu\int_s^{s+t}\big(\norm{\pn}_{L^2}^{1/4}\norm{\Delta\pn}_{L^2}^{7/4}+\norm{\Delta\pn}_{L^2}+\norm{\pn}_{L^2}^{1/8}\norm{\Delta\pn}_{L^2}^{7/8}\norm{\pa^2_y\psi}_{L^2_y}\big)\,\dd \tau\,,
\end{align}
where we used the fact $\norm{\psi}_{L^2_y}\leq C\norm{\partial_y^2\psi}_{L^2_y}$ since $\psi$ is mean free.

On the other hand, we have the following $L^2$ energy estimate,
\begin{align} \label{eq:FirstEnergyEst}
\nonumber
\frac{1}{2}\ddt\norm{\pn}_{L^2}^2+\nu \norm{\Delta \pn}_{L^2}^2&=-\frac{\nu}{2}\int_{\T^3}|\nabla \pn|^2\pn\,\dd x_1\dd x_2\dd y+\frac{\nu}{2L_1L_2}\int_{\T^3}\left(\int_{\T^2}|\nabla\pn|^2\,\dd x_1\dd x_2\right)\pn\,\dd x_1\dd x_2\dd y\\
\nonumber
&\qquad \qquad -\nu\int_{\T^3}\psi\partial_y\pn\pn\,\dd x_1\dd x_2\dd y+\nu\norm{\nabla \pn}_{L^2}^2\\
&\leq C\nu \norm{\nabla\pn}_{L^4}^2\norm{\pn}_{L^2}+C\nu\norm{\psi}_{L^2_y}\norm{\nabla\pn}_{L^4}\norm{\pn}_{L^4}+\nu \norm{\nabla \pn}_{L^2}^2\,.
\end{align}
Use again the Gargliardo-Nirenberg interpolation inequalities in~\eqref{e:soblev1} and
\begin{align}\label{e:soblev}
\norm{\pn}_{L^4}&\leq C\norm{\pn}_{L^2}^{5/8}\norm{\Delta\pn}_{L^2}^{3/8}\,.
\end{align}
Utilizing all these embedding inequalities, we get
\begin{align}
\frac{1}{2}\ddt\norm{\pn}_{L^2}^2+\nu \norm{\Delta \pn}_{L^2}^2
&\leq C\nu\norm{\pn}_{L^2}^{5/4}\norm{\Delta\pn}_{L^2}^{7/4}+C\nu\norm{\psi}_{L^2_y}\norm{\pn}_{L^2}^{3/4}\norm{\Delta \pn}_{L^2}^{5/4}\nonumber\\
&\quad+\nu\norm{\pn}_{L^2}^{1/4}\norm{\Delta\pn}_{L^2}^{7/4}\,.
\end{align}
Finally, applying Young's inequality, we obtain
\begin{align}\label{e:energypn}
\ddt\norm{\pn}_{L^2}^2+\nu \norm{\Delta \pn}_{L^2}^2
&\leq C\nu \norm{\pn}_{L^2}^{10}+C\nu\norm{\pn}_{L^2}^2+C\nu\norm{\pn}_{L^2}^2\norm{\psi}_{L^2_y}^{8/3}\,.
\end{align}

In view of~\eqref{e:duhamel},~\eqref{e:energypn} and decaying property of the semigroup the $e^{-t H_\nu}$ in Corollary~\ref{c:decay},  for all sufficiently small times $t\geq s\geq 0$ we can assume that
\begin{enumerate} [label=(H\arabic*), ref=(H\arabic*)]
\item \label{i:bootstrap1} $\norm{\pn(t)}_{L^2}\leq 8 \e^{-\lambda_\nu( t-s)/4}\norm{\pn(s)}_{L^2}$,
\item  \label{i:bootstrap2} $\nu \int_s^t \norm{\Delta\pn(\tau)}_{L^2}^2\,\dd \tau\leq  4 \norm{\pn(s)}_{L^2}^2$.
\end{enumerate}
We refer to \ref{i:bootstrap1}-\ref{i:bootstrap2} as the {\em bootstrap assumptions}. Let $t_0>0$ be the maximal time such that the bootstrap assumptions above hold on $[0,t_0]$. We will later focus on proving by choosing $\nu_0$ sufficiently small, then for all $\nu\leq \nu_0$, we could always have the following refined estimates:
\begin{enumerate} [label=(B\arabic*), ref=(B\arabic*)]
\item \label{i:bootstrap3} $\norm{\pn(t)}_{L^2}\leq 4 \e^{-\lambda_\nu( t-s)/4}\norm{\pn(s)}_{L^2}$,
\item  \label{i:bootstrap4} $\nu \int_s^t \norm{\Delta\pn(\tau)}_{L^2}^2\,\dd \tau\leq  2 \norm{\pn(s)}_{L^2}^2$,
\end{enumerate}
for all $0\leq s\leq t\leq t_0$. We refer to \ref{i:bootstrap3}-\ref{i:bootstrap4} as the {\em bootstrap estimates}.

 Assuming~\ref{i:bootstrap1} and~\ref{i:bootstrap2} , we can achieve suitable bounds of $\psi$ as long as the restricted domain satisfies no growing modes. We state it in the next lemma.

\begin{lemma}\label{l:linftyl2}
Let the domain $\T^3=[0,L_1]\times[0, L_2]\times[0, L_3]$ satisfy $L_3< 2\pi$. Assume the bootstrap assumptions~\ref{i:bootstrap1} and~\ref{i:bootstrap2}.   There exists a $\nu$-independent constant $C_1=C_1(L_3, \norm{\phi_0}_{L^2}, \norm{\psi(0)}_{L^2})$ satisfying
\begin{align}\label{e:partialpsi}
\norm{\psi(t)}_{L^2_y}^2+\nu\int_0^t\norm{\pa^2_y\psi(s)}_{L^2_y}^2\,\dd s\leq C_1\,,
\end{align}
for all $t\in [0,t_0]$.
\end{lemma}

\begin{proof}
Applying the energy estimate to~\eqref{e:psi} and using Poincar\'e inequality , we get
\begin{align}\label{e:psi2}
\frac{1}{2}\ddt\norm{\psi}_{L^2_y}^2+\nu \norm{\pa^2_y\psi}_{L^2_y}^2&=\nu\norm{\partial_y\psi}_{L^2_y}^2-\frac{\nu}{2L_1L_2}\int_{\T^1}\int_{\T^2}(\pa_y\nabla \pn \cdot \nabla \pn)\psi\,\dd x_1\dd x_2\dd y\\
&\leq \nu\Big(\frac{L_3}{2\pi}\Big)^2\norm{\partial^2_y\psi}_{L^2_y}^2+\frac{\nu}{2L_1L_2} \norm{\nabla \pn}_{L^4}\norm{\Delta \pn}_{L^2}\norm{\psi}_{L^4_y}\,.
\end{align}
We recall the Gagliardo-Nirenberg inequalities:
\begin{align*}
&\norm{\psi}_{L^4_y}\leq \norm{\psi}_{L_y^2}^{7/8}\norm{\partial_y^2\psi}_{L^2_y}^{1/8}\,,\\
&\norm{\nabla \pn}_{L^4}\leq C\norm{\pn}_{L^2}^{1/8}\norm{\Delta\pn}_{L^2}^{7/8}.
\end{align*}
Now estimate~\eqref{e:psi2} becomes
\begin{align*}
\frac{1}{2}\ddt\norm{\psi}_{L^2_y}^2+\nu \Big(1-\big(\frac{L_3}{2\pi}\big)^2\Big)\norm{\pa^2_y\psi}_{L^2_y}^2
&\leq C\nu\norm{\pn}_{L^2}^{1/8}\norm{\Delta \pn}_{L^2}^{15/8}\norm{\psi}_{L^2_y}^{7/8}\norm{\pa^2_y\psi}_{L^2_y}^{1/8}\,.
\end{align*}
Applying Young's inequality, it then yields 
\begin{align}
\label{e:psitmp1}
\nonumber
\ddt\norm{\psi}_{L^2_y}^2+\nu\Big(1-\big(\frac{L_3}{2\pi}\big)^2\Big) \norm{\pa^2_y\psi}_{L^2_y}^2&\leq  C\nu \norm{\pn}_{L^2}^{2/15}\norm{\Delta\pn}_{L^2}^{2}\norm{\psi}_{L^2_y}^{14/15} \\
&\leq C\nu \norm{\pn}_{L^2}^{2/15}\norm{\Delta\pn}_{L^2}^2+C\nu \norm{\pn}_{L^2}^{2/15}\norm{\Delta\pn}_{L^2}^2\norm{\psi}_{L^2_y}^2\,.
\end{align}
We then define the integrating factor $\mu=\exp\big( -C\nu \int_0^t\norm{\pn}_{L^2}^{2/15}\norm{\Delta \pn}_{L^2}^2\,ds\big)$ and apply the bootstrap assumptions~\ref{i:bootstrap1} and~\ref{i:bootstrap2} to get
\begin{align}\label{e:psibd}
\nonumber
\norm{\psi(t)}_{L^2_y}^2&\leq  C\nu\mu^{-1}\int_0^t\norm{\pn}_{L^2}^{2/15}\norm{\Delta \pn}_{L^2}^2\,ds+ \mu^{-1}\norm{\psi(0)}_{L^2_y}^2\\
&\leq  Ce^{C\norm{\phi_0}_{L^2}^{32/15}}\norm{\phi_0}_{L^2}^{32/15}+e^{C\norm{\phi_0}_{L^2}^{32/15}}\norm{\psi(0)}_{L^2_y}^{2}\,.
\end{align}
 Using~\eqref{e:psibd} to~\eqref{e:psitmp1}, we get the desired result~\eqref{e:partialpsi}. This finishes the proof.
\end{proof}

The bootstrap estimates~\ref{i:bootstrap3} and~\ref{i:bootstrap4} will be achieved through a couple of lemmas. We will postpone the proof and assume the fact that bootstrap estimates~\ref{i:bootstrap3} and~\ref{i:bootstrap4} hold on $[0, t_0]$. We next prove Theorem~\eqref{t:KS} to get the global well-posedness of the advective Kuramoto Sivashinsky system.

\begin{proof}[Proof of Theorem \ref{t:KS}]
First from the bootstrap estimates~\ref{i:bootstrap3}, ~\ref{i:bootstrap4} and the definition of $t_0$, we must have $t_0=\infty$. Hence it holds that $\pn\in L^\infty([0,\infty);L^2(\TT^3))\cap L^2([0,\infty);H^2(\T^3))$. Recall Lemma~\ref{l:linftyl2}, we have $\psi\in L_{loc}^\infty([0,\infty);L^2(\T^1))\cap L_{loc}^2([0,\infty);H^2(\T^1))$. Applying the triangle and Poincar\'e's inequalities, we then get $\Delta\phi\in L_{loc}^2([0,\infty);L^2(\T^3))$. We further denote
\begin{align}\label{eq:defphibar}
\bar \phi=\frac{1}{L_1L_2L_3}\int_{\T^3}\phi(x_1,x_2,y)\,\dd x_1\dd x_2\dd y= \frac{1}{L_3}\int_{\T^1}\pz\,\dd y\,.
\end{align}
Then $\bar \phi$ satisfies
\begin{align}
\partial_t\bar \phi&=-\frac{\nu}{2L_1L_2L_3}\int_{\T^3}|\nabla \pn+\nabla\pz|^2\,\dd x_2\dd x_2\dd y =-\frac{\nu}{2L_1L_2L_3}\int_{\T^3} |\nabla \pn|^2\,\dd x_1\dd x_2\dd y-\frac{\nu}{2L_3}\int_{\T^1}|\psi|^2\,\dd y\,.
\end{align}
By integrating the above equation and applying estimate~\eqref{e:partialpsi} and~\ref{i:bootstrap4}, we obtain $\bar\phi\in L_{loc}^\infty ([0,\infty))$. Recall Lemma~\ref{l:linftyl2}, we have
$\psi\in L_{loc}^\infty([0,\infty);L^2(\T^1))$. Hence, we get  $\pz\in L_{loc}^\infty([0,\infty);L^2(\T^1))$ by the Poincar\'e inequality, which then implies $\phi\in L_{loc}^\infty([0,\infty);L^2(\T^3))$.
Finally since $\nabla^2 \phi=\nabla^2 \pn +\de_y \psi$, we obtain $\phi \in L^2_{loc}([0,\infty);H^2(\T^3))$.
\end{proof}

For the rest of this section, we will focus on proving the booststrap estimates~\ref{i:bootstrap3} and~\ref{i:bootstrap4}. We prove~\ref{i:bootstrap4} first.

\begin{lemma}\label{lem:B2}
Let the domain $\T^3=[0,L_1]\times[0, L_2]\times[0, L_3]$ satisfy $L_3< 2\pi$. Assume the bootstrap assumptions~\ref{i:bootstrap1} and~\ref{i:bootstrap2}. There exists $\nu_0=\nu_0(\phi_0)$ with the following property: for any $0\leq s\leq t\leq t_0$ and for any $\nu\leq \nu_0$,
 it holds that
\begin{align}\label{e:B2}
\nu \int_s^t \norm{\Delta\pn(\tau)}_{L^2}^2\,\dd s\leq  2\norm{\pn(s)}_{L^2}^2\,.
\end{align}
\end{lemma}

\begin{proof}
 Using the bootstrap assumptions, Lemma~\ref{l:linftyl2} and the energy estimate~\eqref{e:energypn}, one obtains
\begin{align}
\nu\int_s^t\norm{\Delta\pn(\tau)}_{L^2}^2\,\dd \tau&\leq \norm{\pn(s)}_{L^2}^2+C\nu\int_s^t \norm{\pn(\tau)}_{L^2}^{10}+(1+C_1^{4/3})\norm{\pn(\tau)}_{L^2}^2\,\dd \tau\notag\\
&\leq \norm{\pn(s)}_{L^2}^2+C\nu\int_s^t (1+C_1^{4/3}+\norm{\pn(0)}_{L^2}^8)e^{-\lambda_\nu (\tau -s)/2}\norm{\pn(s)}_{L^2}^2\,d\tau   \\
&\leq \norm{\pn(s)}_{L^2}^2+\frac{\nu}{\lambda_\nu}C(\norm{\pn(0)}_{L^2}^8+1+C_1^{4/3})\norm{\pn(s)}_{L^2}^2\,.
\label{bd:L2L2Delta}
\end{align}
Since $\nu/\lambda_\nu\to 0$ as $\nu\to 0$, we choose $\nu_0$ satisfying
$$
    \frac{\nu_0}{\lambda_{\nu_0}} \leq \frac{1}{C(1+C_1^{4/3}+\norm{\pn(0)}_{L^2}^8)}\,.
$$
We then take any $\nu<\nu_0$ in ~\eqref{bd:L2L2Delta} and obtain~\eqref{e:B2} as desired.
\end{proof}

We are left to prove \ref{i:bootstrap3}, which we accomplish in different steps. We first prove that with the help of the planar helical flow, the energy dissipates much faster and hence we will have a constant fraction decay of $\norm{\pn}_{L^2}$ after a fixed length of time. We state this in the next lemma.

\begin{lemma}\label{l:decaytau}
Let the domain $\T^3=[0,L_1]\times[0, L_2]\times[0, L_3]$ satisfy $L_3< 2\pi$. Assume the bootstrap assumptions~\ref{i:bootstrap1} and~\ref{i:bootstrap2}, and take  $\tau^*=4/\lambda_\nu$.
There exists $\nu_0=\nu_0(\norm{\pn(0)}_{L^2})$, such that for any $s\in [0,t_0]$ satisfying $s+\tau^*\leq t_0$ and for any $\nu\leq \nu_0$,
 \begin{align}
 \norm{\pn(\tau^*+s)}_{L^2}\leq \frac{1}{\e}\norm{\pn(s)}_{L^2}\,.
 \end{align}
\end{lemma}

\begin{proof}
We assume that $\nu_0$ is small enough so that Lemma \ref{lem:B2} holds.
By the definition of $\tau^*$, one has
\begin{align}
\norm{\cS_{\tau^*}(\pn(s))}_{L^2}\leq \frac{5}{\e^4}\norm{\pn(s)}_{L^2}\leq \frac{1}{\e^2}\norm{\pn(s)}_{L^2}\,.
\end{align}
Applying this inequality in~\eqref{e:duhamel} yields
\begin{align}
\norm{\pn(\tau^*+s)}_{L^2}&\leq \frac{1}{\e^2}\norm{\pn(s)}_{L^2}+C\nu\int_s^{\tau^*+s}\big(\norm{\pn}_{L^2}^{1/4}\norm{\Delta\pn}_{L^2}^{7/4}+\norm{\Delta\pn}_{L^2}+\norm{\pn}_{L^2}^{1/8}\norm{\Delta\pn}_{L^2}^{7/8}\norm{\pa^2_y\psi}_{L^2_y}\big)\,\dd \tau \nonumber\\
&\leq  \frac{1}{\e^2}\norm{\pn(s)}_{L^2}+C\Big(\nu\int_s^{\tau^*+s}\norm{\pn(\tau)}_{L^2}^2\,\dd \tau\Big)^{1/8}\Big(\nu\int_s^{\tau^*+s}\norm{\Delta\pn(\tau)}_{L^2}^2\,\dd \tau\Big)^{7/8}\nonumber\\
&\qquad +C\Big(\nu\int_s^{\tau^*+s}\norm{\Delta\pn(\tau)}_{L^2}^2\,\dd \tau\Big)^{1/2}(\nu\tau^*)^{1/2}\nonumber\\
&\qquad+C\Big(\nu\int_s^{\tau^*+s}\norm{\Delta\pn(\tau)}_{L^2}^2\,\dd \tau\Big)^{7/16}
\Big(\nu\int_s^{\tau^*+s}\norm{\pa^2_y\psi(\tau)}_{L^2}^2\,\dd \tau\Big)^{1/2}\Big(\nu\int_s^{\tau^*+s}\norm{\pn(\tau)}_{L^2}^2\,\dd \tau\Big)^{1/16}\,.
\end{align}
Using the bootstrap assumptions \ref{i:bootstrap1}-\ref{i:bootstrap2}   and Lemma~\ref{l:linftyl2}, it then follows that
\begin{align}\label{e:ldecaytmp}
\norm{\pn(\tau^*)}_{L^2}&\leq \frac{1}{\e^2}\norm{\pn(s)}_{L^2}+C(\nu\tau^*)^{1/8}\norm{\pn(s)}_{L^2}^2+C(\nu\tau^*)^{1/2}\norm{\pn(s)}_{L^2}+C\sqrt{C_1}(\nu\tau^*)^{1/16}\norm{\pn(s)}_{L^2}\nonumber\\
&\leq \frac{1}{\e^2}\norm{\pn(s)}_{L^2}+C(\nu\tau^*)^{1/16}\big(\norm{\pn(s)}_{L^2}+\sqrt{C_1}+1\big)\norm{\pn(s)}_{L^2}\,,
\end{align}
where we used the fact that $\nu\tau^*\ll 1$ for $\nu_0 $ sufficiently small. We further note that by the bootstrap assumption~\ref{i:bootstrap1}, it holds that $\norm{\pn(s)}_{L^2}\leq 8\norm{\pn(0)}_{L^2}$. Hence the lemma is proved by choosing
 \begin{align}\label{e:nu03}
 \frac{1}{\e^2} +C(4\nu_0\lambda_{\nu_0}^{-1})^{1/16}\big(8\norm{\pn(0)}_{L^2}+\sqrt{C_1}+1\big)\leq \frac{1}{\e}\,.
  \end{align}
 
\end{proof}

Finally, we are ready to prove that the bootstrap assumption~\ref{i:bootstrap1} can be refined.
\begin{lemma}\label{lem:B1}
Let the domain $\T^3=[0,L_1]\times[0, L_2]\times[0, L_3]$ satisfy $L_3< 2\pi$. Assume the bootstrap assumptions~\ref{i:bootstrap1} and~\ref{i:bootstrap2}. There exists $\nu_0=\nu_0(\norm{\pn(0)}_{L^2})$ such that for any $ 0\leq s\leq t\leq t_0$ and for any $\nu\leq \nu_0$,
 it holds that
\begin{align}
\norm{\pn(t)}_{L^2}\leq 4\e^{-\lambda_\nu (t-s)/4}\norm{\pn(s)}_{L^2}\,.
\end{align}
\end{lemma}

\begin{proof}
We take $\nu_0 =\nu_0(\norm{\pn(0)}_{L^2})$ so that  Lemma  \ref{lem:B2}-\ref{l:decaytau} apply.
By Lemma~\ref{l:decaytau}, one has
\begin{align}
\norm{\pn(s+n\tau^*)}_{L^2}\leq \e^{-n}\norm{\pn(s)}_{L^2}\leq \norm{\pn(s)}_{L^2}\,,\quad\text{for any $n\in \ZZ_+$.}
\end{align}
For any $t\in [s, t_0]$, there exists $n$ such that $t\in[s+n\tau^*,~s+(n+1)\tau^*)$. Recalling Lemma~\ref{l:linftyl2}, the energy estimate~\eqref{e:energypn}, and the bootstrap assumption \ref{i:bootstrap1}, then for some positive
$C_2=C_2(\norm{\pn(0)}_{L^2},\norm{\psi(0)}_{L^2_y})$ it holds that
\begin{align}
\ddt\norm{\pn}_{L^2}^2&\leq C\nu\norm{\pn}_{L^2}^{10}+CC_1^{4/3} \nu \norm{\pn}_{L^2}^2\notag\\
&\leq C(\norm{\pn(0)}_{L^2}^8+C_1^{4/3})\nu \norm{\phi_{\neq}}_{L^2}^2\notag\\
&\leq \nu C_2\norm{\pn}_{L^2}^2\,.\label{e:stima}
\end{align}
By choosing 
\begin{align}
C_2\nu\tau^*\leq \ln 2\,,
\end{align}
which is equivalent to asking for $\nu\leq \nu_0$ and
\begin{align}
\frac{\nu_0}{\lambda_{\nu_0}}\leq \frac{\ln 2}{4C_2}\,, 
\end{align}
one has for all $t_1\in[0,t_0]$ and $\tau\in [t_1, t_1+\tau^*]\cap [0,t_0]$,
 \begin{align}
\norm{\pn(\tau)}_{L^2}\leq \sqrt 2\norm{\pn(t_1)}_{L^2}\,.
\end{align}
It then follows that
\begin{align*}
\norm{\pn(t)}_{L^2}\leq \sqrt 2 \norm{\pn(s+n\tau^*)}_{L^2}\leq \sqrt 2 \e^{-n}\norm{\pn(s)}_{L^2}\leq \sqrt 2 \e^{1-(t-s)/\tau^*}\norm{\pn(s)}_{L^2}\\
\leq 4\e^{-\lambda_\nu (t-s)/4}\norm{\pn(s)}_{L^2}\,.
\end{align*}
\end{proof}

\section{Global existence of the 3D Keller-Segel equation with planar helical flow}\label{s:KSegel}
We prove the global  existence of the classical solution of the solution of the three dimensional advective Keller-Segel equation~\eqref{e:KSegel} with any initial data.
According to the regularity criterion of solution (see~\cite{Kiselev.2016}), we only  need to get the global $L^2$ estimate of solution.
 
Let $\rho$ be the solution of the three dimensional Keller-Segel equation~\eqref{e:KSegel}. Recall the decomposition of functions defined in~\eqref{e:decompose}, the Keller-Segel equation can be decomposed as
\begin{equation}\label{e:rz}
\begin{cases}
\partial_t\rz-\nu\partial_{yy}\rz+\nu\partial_y\big(\rz \partial_y \cz\big)+\nu\langle\nabla\cdot(\rho_{\neq} \nabla c_{\neq} )\rangle=0,\\
-\partial_{yy} \cz=\rz-\bar\rho\,,
\end{cases}
\end{equation}
and
\begin{equation}\label{e:rn}
\begin{cases}
\begin{aligned}
\partial_t\rho_{\neq}+\sin y \partial_{x_1}\rho_{\neq}&+\cos y \partial_{x_2}\rho_{\neq}-\nu\Delta\rho_{\neq}+\nu\nabla\rz\cdot\nabla c_{\neq}+\nu\nabla\rho_{\neq}\cdot\nabla \cz\\
&+\nu\left(\nabla\cdot(\rho_{\neq} \nabla c_{\neq})\right)_{\neq}-\nu \rz\rho_{\neq}-\nu \rho_{\neq}(\rz-\overline{\rho})=0,\\
-\Delta c_{\neq}=\rho_{\neq}\,.
\end{aligned}
\end{cases}
\end{equation}

We again establish the global well-posedness based on the bootstrap argument. Inspired by~\cite{Bedrossian.2017}, we list the bootstrap assumptions as below:
\begin{assumption}\label{ass:KSegelBOOT}
Let $ C_{L^2}, C_{\dot{H}^1}$ and $C_\infty$ be the positive constants which are independent of $\nu$ and $t$. From the standard energy estimate, the following estimates hold at least within a short time of order $O(1/\lambda_\nu)$. Here we denote $t_0$ as the maximum time such that the following estimates hold.
\begin{itemize}
  \item [(A-1)] Nonzero mode $L^2\dot{H}^1$ estimate: for any $0\leq s\leq t\leq t_0$
  \begin{equation}\label{eq:L2H1}
  \nu\int_{s}^{t}\|\nabla \rho_{\neq}\|^2_{L^2}d\tau\leq 128 e^{-\lambda_\nu s/4}\|\rho_{0}\|^2_{L^2},
  \end{equation}
  \item [(A-2)]Nonzero mode enhanced dissipation estimate: for any $t\in [0,t_0]$
   \begin{equation}\label{eq:L2}
  \| \rho_{\neq}(t)\|_{L^2}^2\leq 32 e^{-\lambda_\nu t/4}\|\rho_{0}\|_{L^2}^2,
  \end{equation}
  \item [(A-3)] Uniform in time $L^\infty L^2_y$ estimate on the zero mode
  \begin{equation}\label{eq:LinfL2}
  \| \rz-\overline{\rho}\|_{L^\infty(0, t_0;L_y^2)}\leq 4C_{L^2},
  \end{equation}
  \item [(A-4)]Uniform in time $L^\infty \dot{H}^1_y$ estimate on the zero mode
  \begin{equation}\label{eq:LinfH1}
  \|\partial_y \rz\|_{L^\infty(0, t_0;L_y^2)}\leq 4C_{\dot{H}^1},
  \end{equation}
 \item [(A-5)] $L^\infty$ estimate of solution
 \begin{equation}\label{eq:Linf}
  \|\rho\|_{L^\infty(0, t_0;L^\infty)}\leq 4C_\infty.
  \end{equation}
\end{itemize}
\end{assumption}

We aim to show $t_0=\infty$. This is achieved through the bootstrap argument. To be specific, we will prove the following refined estimates hold on $[0, t_0]$ by choosing proper $\nu$ .
\begin{itemize}
  \item [(B-1)] Nonzero mode $L^2\dot{H}^1$ estimate: for any $0\leq s\leq t\leq t_0$
  \begin{equation}\label{eq:L2H1new}
  \nu\int_{s}^{t}\|\nabla \rho_{\neq}\|^2_{L^2}d\tau\leq 64 e^{-\lambda_\nu s/4}\|\rho_{0}\|^2_{L^2},
  \end{equation}
  \item [(B-2)]Nonzero mode enhanced dissipation estimate: for any $t\in [0,t_0]$
   \begin{equation}\label{eq:L2new}
  \| \rho_{\neq}(t)\|_{L^2}^2\leq 16 e^{-\lambda_\nu t/4}\|\rho_{0}\|_{L^2}^2,
  \end{equation}
  \item [(B-3)] Uniform in time $L^\infty L^2_y$ estimate on the zero mode
  \begin{equation}\label{eq:LinfL2new}
  \| \rz-\overline{\rho}\|_{L^\infty(0, t_0;L_y^2)}\leq 2C_{L^2},
  \end{equation}
  \item [(B-4)]Uniform in time $L^\infty \dot{H}^1_y$ estimate on the zero mode
  \begin{equation}\label{eq:LinfH1new}
  \|\partial_y \rz\|_{L^\infty(0, t_0;L_y^2)}\leq 2C_{\dot{H}^1},
  \end{equation}
 \item [(B-5)] $L^\infty$ estimate of solution
 \begin{equation}\label{eq:Linfnew}
  \|\rho\|_{L^\infty(0, t_0;L^\infty)}\leq 2C_\infty.
  \end{equation}
\end{itemize}
Once the estimates (B-1)$-$(B-5) are established, we obtain $t_0=\infty$. The global existence of the three dimensional Keller-Segel thus follows directly. Hence to prove Theorem~\ref{t:KSegel}, we only need to prove the refined estimates (B-1)$-$(B-5).

One well known fact is that the $L^1$ norm of the density is conserved. For the notational convenience,  we define
\begin{equation}\label{e:L1}
M:=\norm{\rho}_{L^1}=\norm{\rho_0}_{L^1}\,.
\end{equation}

First, we will show that the estimate~\eqref{eq:L2H1} can be refined by choosing $\nu_0$ small enough.
\begin{lemma}[ $L^2\dot{H}^1$ estimate of $\rho_{\neq}$ ]\label{lem:L2H1}
Assume Assumption~\eqref{ass:KSegelBOOT} holds. Let $\rz$ and $\rho_{\neq}$ be the solutions of equations (\ref{e:rz}) and (\ref{e:rn}).  There exists $\nu_0=\nu(\rho_0)$, such that for any $\nu<\nu_0$, $0\leq s\leq t\leq t_0$, one has
\begin{equation}
\nu\int_{s}^{t}\|\nabla \rho_{\neq}\|^2_{L^2}d\tau\leq 64 e^{-\lambda_\nu s/4}\|\rho_{0}\|^2_{L^2}.
\end{equation}
\end{lemma}

\begin{proof}
Multiplying both sides of (\ref{e:rz}) by $\rho_{\neq}$ and integrating over $\mathbb{T}^3$, one obtains
\begin{align*}
\nonumber
\frac{1}{2}\frac{d}{dt}\|\rho_{\neq}\|^2_{L^2}+\nu\|\nabla\rho_{\neq}\|^2_{L^2}&=-\nu\int_{\mathbb{T}^3}\nabla\rz\cdot\nabla c_{\neq}\rho_{\neq}dx_1dx_2dy-\nu\int_{\mathbb{T}^3}\nabla\rho_{\neq}\cdot\nabla \cz\rho_{\neq}dx_1dx_2dy\\
\nonumber
&\quad-\nu\int_{\mathbb{T}^3}\left(\nabla\cdot(\rho_{\neq} \nabla c_{\neq})\right)_{\neq}\rho_{\neq}dx_1dx_2dy+\nu \int_{\mathbb{T}^3}\rz\rho_{\neq}\rho_{\neq}dx_1dx_2dy\\
&\quad+\nu \int_{\mathbb{T}^3}\rho_{\neq}(\rz-\overline{\rho})\rho_{\neq}dx_1dx_2dy\,.
\end{align*}
Treating the right hand side with H\"{o}lder's inequality and integrating by parts, one has
\begin{align}\label{e:rzenergy}
\nonumber
\frac{1}{2}\frac{d}{dt}\|\rho_{\neq}\|^2_{L^2}+\nu\|\nabla\rho_{\neq}\|^2_{L^2}&\leq C\nu \norm{\de_y\rz}_{L^2_y}\norm{\nabla\cn}_{L^4}\norm{\rn}_{L^4}+C\nu\norm{\rz-\bar \rho}_{L^\infty_y}\norm{\rn}_{L^2}^2+C\nu\norm{\rn}_{L^3}^3\\
&\quad+C\nu \norm{\nabla \rn}_{L^2}\norm{\nabla \cn}_{L^4}\norm{\rn}_{L^4}+C\nu\norm{\rz}_{L^\infty_y}\norm{\rn}_{L^2}^2\,.
\end{align}
Recall the Gagliardo-Nirenberg inequalities:
\begin{equation}\label{e:sobolev}
\begin{aligned}
\norm{\nabla \cn}_{L^4}&\leq C\norm{\Delta\cn}_{L^2}^{7/8}\norm{\cn}_{L^2}^{1/8}\leq C\norm{\Delta \cn}_{L^2}\,,\\
\norm{\rn}_{L^4}&\leq C\norm{ \rn}_{L^2}^{1/4}\norm{\nabla\rn}_{L^2}^{3/4}\,,\\
\norm{\rn}_{L^3}&\leq C\norm{\rn}_{L^2}^{1/2}\norm{\nabla\rn}_{L^2}^{1/2}\,.
\end{aligned}
\end{equation}
Applying these inequalities in~\eqref{e:rzenergy}, we obtain
\begin{align*}
\frac{1}{2}\frac{d}{dt}\|\rho_{\neq}\|^2_{L^2}+\nu\|\nabla\rho_{\neq}\|^2_{L^2}&\leq C\nu \norm{\de_y\rz}_{L^2_y}\norm{\rn}_{L^2}^{5/4}\norm{\nabla\rn}_{L^2}^{3/4}+C\nu\norm{\rz-\bar \rho}_{L^\infty_y}\norm{\rn}_{L^2}^2\\
&\quad+C\nu\norm{\rn}_{L^2}^{3/2}\norm{\nabla\rn}_{L^2}^{3/2}+C\nu\norm{\rn}_{L^2}^{5/4}\norm{\nabla\rn}_{L^2}^{7/4}+C\nu\norm{\rz}_{L^\infty_y}\norm{\rn}_{L^2}^2\,.
\end{align*}
We then use Young's inequality to get
\begin{align}\label{e:rnenergy1}
\nonumber
\frac{d}{dt}\|\rho_{\neq}\|^2_{L^2}+\nu\|\nabla\rho_{\neq}\|^2_{L^2}&\leq C\nu\Big(\norm{\de_y\rz}_{L^2_y}^{8/5}+\norm{\rz-\bar \rho}_{L^\infty_y}+\norm{\rz}_{L^\infty_y} \Big)\norm{\rn}_{L^2}^2\\
&\quad+C\nu\norm{\rn}_{L^2}^6+C\nu\norm{\rn}_{L^2}^{10}\,.
\end{align}
According to the Assumption (\ref{ass:KSegelBOOT}), we deduce that for any $t\in[0, t_0]$, one has
\begin{align}\label{e:rnenergy2}
\frac{d}{dt}\|\rho_{\neq}\|^2_{L^2}+\nu\|\nabla\rho_{\neq}\|^2_{L^2}&\leq
C\nu\Big( C_{\dot H^1}^{8/5}+C_{\infty}+\norm{\rho_0}_{L^2}^4+\norm{\rho_0}_{L^2}^8\Big)\norm{\rn}_{L^2}^2\,.
\end{align}
For any $0\leq s\leq t\leq t_0$, integrating from $s$ to $t$, equation~\eqref{e:rnenergy2} becomes
\begin{equation*}
\begin{aligned}
\|\rho_{\neq}(t)\|^2_{L^2}&+\nu\int_{s}^{t}\|\nabla\rho_{\neq}\|^2_{L^2}dt\leq \|\rho_{\neq}(s)\|^2_{L^2}+\frac{\nu}{\lambda_\nu}C\Big( C_{\dot H^1}^{8/5}+C_{\infty}+\norm{\rho_0}_{L^2}^4+\norm{\rho_0}_{L^2}^8\Big)e^{-\lambda_\nu s/4}\norm{\rho_0}_{L^2}^2\\
&\leq 32 e^{-\lambda_\nu s/4}\norm{\rho_0}_{L^2}^2+\frac{\nu}{\lambda_\nu}C\Big( C_{\dot H^1}^{8/5}+C_{\infty}+\norm{\rho_0}_{L^2}^4+\norm{\rho_0}_{L^2}^8\Big)e^{-\lambda_\nu s/4}\norm{\rho_0}_{L^2}^2\,.
\end{aligned}
\end{equation*}
By choosing $\nu_0$ satisfying
\begin{align*}
\frac{\nu_0}{\lambda_{\nu_0}}\leq \frac{32}{C\Big( C_{\dot H^1}^{8/5}+C_{\infty}+\norm{\rho_0}_{L^2}^4+\norm{\rho_0}_{L^2}^8\Big)}\,,
\end{align*}
we get the desired result.
\end{proof}

We then establish the refined estimate of $\norm{\rz-\bar \rho}_{L^2_y}$ via the standard estimate.

\begin{lemma}[ $L^\infty L_y^2$ estimate of $\rz$ ]\label{lem:LinfL2}
Let Assumption \ref{ass:KSegelBOOT} hold and $\rz$ and $\rn$ be the solutions of equations (\ref{e:rz}) and (\ref{e:rn}).  There exists $\nu_0=\nu(\rho_0)$, such that for any $\nu<\nu_0$, one has
\begin{equation}\label{eq:4.31}
\| \rz-\overline{\rho}\|_{L^\infty(0, t_0;L_y^2})\leq 2C_{L^2}.
\end{equation}
\end{lemma}

\begin{proof}
Multiplying both sides of (\ref{e:rz}) by $\rz-\bar{\rho}$, integrating over $\mathbb{T}$, using H\"older's inequality and integration by parts, one has
\begin{equation*}
\begin{aligned}
\frac{1}{2}\frac{d}{dt}\|\rz-\overline{\rho}\|^2_{L^2_y}+\nu\|\partial_y \rz\|^2_{L^2_y}&=-\nu \int_{\T}\partial_y(\rz \partial_y \cz)(\rz-\overline{\rho})dy-\nu\int_{\T}\langle\nabla\cdot(\rho_{\neq} \nabla c_{\neq})\rangle(\rz-\overline{\rho})dy\\
&=\nu\int_\T(\rz-\bar \rho)\de_y\cz\de_y\rz\,dy+\nu\int_\T\bar \rho\de_y\cz\de_y\rz\,dy-\overline{\rho})dy\\
&\quad-\nu\int_{\T}\langle\nabla\cdot(\rho_{\neq} \nabla c_{\neq})\rangle(\rz-\overline{\rho})dy\\
&\leq \nu\norm{\rz-\bar\rho}_{L^2_y}\norm{\de_y\cz}_{L^\infty_y}\norm{\de_y\rz}_{L^2_y}+CM\nu\norm{\de_y\cz}_{L^2_y}\norm{\de_y\rz}_{L^2_y}\\
&\quad+C\nu\norm{\rn}_{L^2}\norm{\nabla \cn}_{L^\infty}\norm{\de_y\rz}_{L^2}\,,
\end{aligned}
\end{equation*}
where we used the fact $\bar\rho\leq CM$ in the last inequality. We next apply the following Gagliardo-Nirenberg and Poincar\'e inequalities 
\begin{align*}
\norm{\de_y\cz}_{L^\infty_y}&\leq C\norm{\de_{yy}\cz}_{L^1_y}\leq C\norm{\rz-\bar \rho}_{L^1_y}\leq CM\,,\\
\norm{\de_y\cz}_{L^2_y}&\leq C\norm{\de_{yy}\cz}_{L^2_y}\leq C\norm{\rz-\bar\rho}_{L^2_y}\,,\\
\norm{\nabla\cn}_{L^\infty}&\leq C\norm{\nabla^2 \cn}_{L^4}= C\norm{\rn}_{L^4}\leq C\norm{\rn}_{L^2}^{1/4}\norm{\nabla\rn}_{L^2}^{3/4}\,,
\end{align*}
to get
\begin{align*}
\frac{1}{2}\frac{d}{dt}\|\rz-\overline{\rho}\|^2_{L^2_y}+\nu\|\partial_y \rz\|^2_{L^2_y}&\leq CM\nu\norm{\rz-\bar \rho}_{L^2_y}\norm{\de_y\rz}_{L^2_y}+C\nu\norm{\rn}_{L^2}^{5/4}\norm{\nabla\rn}_{L^2}^{3/4}\norm{\de_y\rz}_{L^2}\,.
\end{align*}
Then Young's inequality implies
\begin{align}\label{e:rzenergy}
\frac{d}{dt}\|\rz-\overline{\rho}\|^2_{L^2_y}\leq -\nu\|\partial_y \rz\|^2_{L^2_y}+ CM^2\nu\norm{\rz-\bar \rho}_{L^2_y}^2+C\nu\norm{\rn}_{L^2}^{10}+\frac{\nu}{256}\norm{\nabla\rn}_{L^2}^{2}\,.
\end{align}
We observe that the first term on the right hand side of~\eqref{e:rzenergy} can be treated with the following Gaglirado-Nirenberg inequality and estimated by
\begin{align*}
-\norm{\de_y\rz}_{L^2_y}^2\leq -\frac{\norm{\rz-\bar\rho}_{L^2_y}^6}{C\norm{\rz-\bar\rho}_{L^1_y}^4}\leq -\frac{\norm{\rz-\bar\rho}_{L^2}^6}{CM^4}\,.
\end{align*}
Plugging this in~\eqref{e:rzenergy}, one has
\begin{align}\label{e:rzenergy1}
\nonumber
\frac{d}{dt}\|\rz-\overline{\rho}\|^2_{L^2_y}&\leq -C\nu\frac{\norm{\rz-\bar\rho}_{L^2}^6}{M^4}+ CM^2\nu\norm{\rz-\bar \rho}_{L^2_y}^2+C\nu\norm{\rn}_{L^2}^{10}+\frac{\nu}{256}\norm{\nabla\rn}_{L^2}^{2}\\
&\leq -C\nu\frac{\norm{\rz-\bar\rho}_{L^2_y}^2}{M^4}\Big(\norm{\rz-\bar\rho}_{L^2_y}^4-CM^6\Big)+C\nu\norm{\rn}_{L^2}^{10}+\frac{\nu}{256}\norm{\nabla\rn}_{L^2}^{2}\,.
\end{align}
We then define
\begin{equation*}
G(t)=\int_{0}^{t}\frac{\nu}{256}\|\nabla \rho_{\neq}\|^2_{L^2}+C\nu\|\rho_{\neq}\|^{10}_{L^2}ds\,.
\end{equation*}
Following the assumptions in~\eqref{eq:L2H1} and~\eqref{eq:L2}, it holds that for any $t\in [0, t_0]$, 
\begin{align*}
0\leq G(t)\leq \frac{1}{2}\norm{\rho_0}_{L^2}^2+\frac{\nu }{\lambda_\nu}C\norm{\rho_0}_{L^2}^{10}\leq \norm{\rho_0}_{L^2}^2\,,
\end{align*}
by choosing $\nu\leq \nu_0$ satisfying
\begin{align*}
\frac{\nu_0}{\lambda_{\nu_0}}\leq \frac{1}{2C\big(\norm{\rho_0}_{L^2}^8+1\big)}\,.
\end{align*}
Plugging the estimate for $G(t)$ into the inequality~\eqref{e:rzenergy1}, we obtain
\begin{align}
\frac{d}{dt}\left(\|\rz-\overline{\rho}\|^2_{L^2_y}-G(t)\right)
\leq-C\nu\frac{\|\rz-\overline{\rho}\|^2_{L^2_y}}{M^4}\left(\|\rz-\overline{\rho}\|^2_{L^2_y}-G(t)-\sqrt{C}M^3\right)
\left(\|\rz-\overline{\rho}\|^2_{L^2_y}+\sqrt{C}M^3\right)\,.
\end{align}
This implies for any $t\in[0, t_0]$,
\begin{align*}
\norm{\rz-\rho_0}_{L^2}^2\leq \sqrt C M^3+CM^2+\norm{\rho_0}_{L^2}^2\,.
\end{align*}
By denoting
\begin{equation*}
C^2_{L^2}\triangleq CM^3+CM^2+\|\rho_0\|^2_{L^2}\,,
\end{equation*}
we conclude that 
\begin{equation}
\| \rz-\overline{\rho}\|_{L^\infty(0, t_0;L_y^2})\leq 2C_{L^2}\,.
\end{equation}
\end{proof}

Again, through the standard energy estimate, we get the refined estimate of $\norm{\de_y\rz}_{L^2_y}$ in below.
\begin{lemma}[ $L^\infty \dot{H}^1_y$ estimate of $\rz$ ]\label{lem:LinfH1}
Let Assumption \ref{ass:KSegelBOOT} hold and $\rz$ and $\rn$ be the solutions of equations (\ref{e:rz}) and (\ref{e:rn}).  There exists $\nu_0=\nu(\rho_0)$, such that for any $\nu<\nu_0$, it holds that
 \begin{equation}
 \|\partial_y \rz\|_{L^\infty(0, T^\ast;L_y^2})\leq 2C_{\dot{H}^1}.
  \end{equation}
\end{lemma}

\begin{proof}
Applying  the operator $\partial_y$ to (\ref{e:rz}), it yields that
\begin{equation}\label{e:derz}
\partial_t\partial_y\rz-\nu\partial_{yyy}\rz+\nu\partial_{yy}(\rz\partial_y \cz)+\nu\partial_y\langle\nabla\cdot(\rho_{\neq} \nabla c_{\neq})\rangle=0\,.
\end{equation}
Let us multiply both sides  by $\partial_y\rz$ and integrate  over $\mathbb{T}$. We then get
\begin{align}\label{e:derzenergy}
\nonumber
\frac{1}{2}\frac{d}{dt}\|\partial_y\rz\|^2_{L^2_y}+\nu\|\partial_{yy}\rz\|^2_{L^2_y}&=-\nu\int_{\mathbb{T}}\partial_{yy}(\rz\partial_y \cz)\partial_y\rz dy-\nu\int_{\mathbb{T}}\partial_y\langle\nabla\cdot(\rho_{\neq} \nabla c_{\neq})\rangle \partial_y\rz dy\\
\nonumber
&=\nu\int_{\T}\de_y\rz\de_y\cz\de_{yy}\rz\,dy+\int_{\T}\rz\de_{yy}\cz\de_{yy}\rz\,dy\\
\nonumber
&\quad+\nu\int_{\mathbb{T}}\langle\nabla\cdot(\rho_{\neq} \nabla c_{\neq})\rangle \partial_{yy}\rz dy\\
\nonumber
&\leq C\nu\norm{\de_y\rz}_{L^2_y}\norm{\de_y\cz}_{L^\infty_y}\norm{\de_{yy}\rz}_{L_y^2}+C\nu\norm{\rz}_{L^2_y}\norm{\de_{yy}\cz}_{L^\infty_y}\norm{\de_{yy}\rz}_{L^2_y}\\
&\quad+C\nu\norm{\de_y\rn}_{L^2}\norm{\de_y\cn}_{L^\infty}\norm{\de_{yy}\rz}_{L^2_y}+C\nu\norm{\rn}_{L^2}\norm{\de_{yy}\cn}_{L^\infty}\norm{\de_{yy}\rz}_{L^2_y}\,,
\end{align}
where we used the integration by parts and H\"older' inequality in the above estimate.

Again, we apply the following Gagliardo-Nirenberg inequalities,
\begin{align*}
\norm{\de_y\cz}_{L^\infty_y}&\leq \norm{\de_{yy}\cz}_{L^1_y}\leq C\norm{\rz-\bar\rho}_{L^1_y}\leq CM\,,\\
\norm{\de_{yy}\cz}_{L^\infty_y}&=\norm{\rz-\bar\rho}_{L^\infty_y}\leq \norm{\de_y\rz}_{L^2_y}\,,\\
\norm{\de_y\cn}_{L^\infty}&\leq C\norm{\nabla^2\cn}_{L^4}\leq C\norm{\rn}_{L^4}\,,
\end{align*}
 to get
\begin{align*}
\frac{1}{2}\frac{d}{dt}\|\partial_y\rz\|^2_{L^2_y}+\nu\|\partial_{yy}\rz\|^2_{L^2_y}&\leq CM\nu\norm{\de_y\rz}_{L^2_y}\norm{\de_{yy}\rz}_{L^2_y}+C\nu\norm{\rz}_{L^2_y}\norm{\de_y\rz}_{L^2_y}\norm{\de_{yy}\rz}_{L^2_y}\\
&\quad+C\nu\norm{\nabla\rn}_{L^2}\norm{\rn}_{L^4}\norm{\de_{yy}\rz}_{L^2_y}+C\nu\norm{\rn}_{L^2}\norm{\rn}_{L^\infty}\norm{\de_{yy}\rz}_{L^2_y}\,.
\end{align*}
We then use Young's inequality and Assumption~\eqref{ass:KSegelBOOT} to get
\begin{align}\label{e:derzenergy1}
\nonumber
\frac{d}{dt}\|\partial_y\rz\|^2_{L^2_y}&\leq -\nu\|\partial_{yy}\rz\|^2_{L^2_y}+CM^2\nu\norm{\de_y\rz}_{L^2_y}^2+C\nu\norm{\rz}_{L^2_y}^2\norm{\de_y\rz}_{L^2_y}^2\\
&\quad+C\nu\norm{\nabla \rn}_{L^2}^2\norm{\rn}_{L^4}^2+C\nu\norm{\rn}_{L^2}^2\norm{\rn}_{L^\infty}^2\,.
\end{align}
Define
\begin{equation*}
G(t)=\int_{0}^tC\nu \|\nabla\rho_{\neq}\|^2_{L^2}\|\rho_{\neq}\|^2_{L^4}+C\nu\|\rho_{\neq}\|^2_{L^2}\|\rho_{\neq}\|^2_{L^\infty}ds\,.
\end{equation*}
From Assumption~\eqref{ass:KSegelBOOT}, it holds that for any $t\in [0, t_0]$,
\begin{equation}\label{e:deG}
0\leq G(t) \leq CC_\infty^2(1+\frac{\nu}{\lambda_\nu})\norm{\rho_0}_{L^2}^2\leq 2CC_\infty^2\norm{\rho_0}_{L^2}^2\,,
\end{equation}
where we use the fact that $\nu/\lambda_\nu \ll 1$ when $\nu$ is sufficiently small.
Applying Gagliardo-Nirenberg inequality and (\ref{eq:LinfL2}), one has
\begin{equation}\label{e:derztmp1}
-\|\partial_{yy}\rz\|^{2}_{L^2}\leq -\frac{\|\partial_y\rz\|^4_{L^2}}{C\|\rz-\overline{\rho}\|^{2}_{L^2}}\leq -\frac{\|\partial_y\rz\|^4_{L^2}}{CC^2_{L^2}}.
\end{equation}
By Assumption \ref{ass:KSegelBOOT} and (\ref{e:L1}), it holds that
\begin{equation}\label{e:derztmp2}
\|\rz-\overline{\rho}\|^2_{L^2_y}\leq CC^2_{L^2},\ \ \ \ \|\rz\|^2_{L^2_y}\leq CC^2_{L^2}+CM^2.
\end{equation}
Combining (\ref{e:derzenergy1}), (\ref{e:deG}), (\ref{e:derztmp1}) and~\eqref{e:derztmp2}, one has
\begin{equation}
\begin{aligned}
\frac{d}{dt}(\|\partial_y\rz\|^2_{L^2_y}-G(t))\leq -\frac{C\nu}{C^2_{L^2}+M^2}\|\partial_y\rz\|^2_{L^2_y}\left(\|\partial_y\rz\|^2_{L^2_y}-G(t)-C(C^2_{L^2}+M^2)^2\right)\,.
\end{aligned}
\end{equation}
This implies
\begin{equation}
\|\partial_y\rz\|^2_{L^2_y}\leq \|\partial_y\rho_0\|^2_{L^2_y}+ 2CC_\infty^2\norm{\rho_0}_{L^2}^2+C(C_{L^2}^2+M^2)^2\,.
\end{equation}
Thus, we can choose $C_{\dot{H}^1}>\|\partial_y\rho_0\|^2_{L^2_y}+ 2CC_\infty^2\norm{\rho_0}_{L^2}^2+C(C_{L^2}^2+M^2)^2$, such that for any $t\in [0,t_0]$, one has
\begin{equation*}
\|\partial_y \rz\|_{L^\infty(0, t_0;L_y^2})\leq 2C_{\dot{H}^1}\,.
\end{equation*}
as claimed.
\end{proof}

Since the $L^2$ estimate of solution to equation (\ref{e:KSegel}) is supercritical, by Moser iteration, the $L^\infty$ estimate of solution is obtained.

\begin{lemma}[ $L^\infty L^\infty$ estimate of $\rho$ ]\label{lem:LinfLinf}
Assume  Assumption \ref{ass:KSegelBOOT} hold. Let $\rz$ and $\rho_{\neq}$ be the solutions of equations (\ref{e:rz}) and (\ref{e:rn}).  Then one has
\begin{equation}\label{eq:LemmaLinf}
\|\rho\|_{L^\infty(0, T^\ast;L^\infty})\leq 2C_\infty\,.
\end{equation}
\end{lemma}

\begin{proof}
By Assumption \ref{ass:KSegelBOOT}, such that for any $t\in [0,t_0]$, one has
\begin{equation*}
\|\rho\|_{L^2}\leq C\|\rho_{\neq}\|_{L^2}+C\|\rz-\overline{\rho}\|_{L^2}+C\overline{\rho}\leq C\|\rho_0\|_{L^2}+CC_{L^2}+CM<\infty\,.
\end{equation*}
As $L^2$ norm is supercritical (see~\cite{Jager1992explosions, Kowalczyk2005preventing} for reference),  by Moser-Alikakos iteration~\cite{Alikakos1979lp}, the uniform-in-time bound of the $L^\infty$ norm only depends on the uniform-in-time bound of the $L^2$ norm. Thus the estimate~\eqref{eq:LemmaLinf} follows by choosing $C_\infty$ appropriately.
\end{proof}

Next, we establish the enhanced dissipation estimate of $\rho_{\neq}$ using the fast dissipation property of the underlying semigroup.

\begin{lemma}[Enhanced dissipation estimate of $\rho_{\neq}$ ]\label{lem:L2}
Assume  Assumption \ref{ass:KSegelBOOT} hold. Let $\rz$ and $\rho_{\neq}$ be the solutions of equations (\ref{e:rz}) and (\ref{e:rn}). Then there exist $\nu_0=\nu(\rho_0)$, if $\nu<\nu_0$, one has
\begin{equation}\label{eq:4.73}
\| \rho_{\neq}(t)\|^2_{L^2}\leq 16e^{-\lambda_\nu t/4}\|\rho_{0}\|^2_{L^2}.
\end{equation}
\end{lemma}

\begin{proof}
By Duhamel's principle, the solution of equation (\ref{e:rn}) can be rewritten  as
\begin{equation*}
\begin{aligned}
 \rho_{\neq}(s+t)&=\mathcal{S}_{t}\rho_{\neq}(s)\\
 &-\nu\int_{s}^{s+t}\mathcal{S}_{t+s-\tau}\left(\nabla\rz\cdot\nabla c_{\neq}+\nabla\rho_{\neq}\cdot\nabla \cz+\left(\nabla\cdot(\rho_{\neq} \nabla c_{\neq})\right)_{\neq}- \rz\rho_{\neq}-\rho_{\neq}(\rz-\overline{\rho})\right)d\tau.
 \end{aligned}
\end{equation*}
This implies
\begin{equation*}
\begin{aligned}
\|\rho_{\neq}(t+s)\|_{L^2}&\leq \|\mathcal{S}_{t}\rho_{\neq}(s)\|_{L^2}+C\nu\int_{s}^{s+t}\Big(\|\partial_y\rz\partial_yc_{\neq}\|_{L^2}+\|\partial_y\rho_{\neq}\partial_y \cz\|_{L^2}+\|\left(\nabla\cdot(\rho_{\neq} \nabla c_{\neq})\right)_{\neq}\|_{L^2}\\
&\quad+\|\rz\rho_{\neq}\|_{L^2}+\|\rho_{\neq}(\rz-\overline{\rho})\|_{L^2}\Big)d\tau\\
&\leq \|\mathcal{S}_{t}\rho_{\neq}(s)\|_{L^2}+C\nu\int_{s}^{s+t}\Big(\norm{\de_y\rz}_{L^2_y}\norm{\de_y \cn}_{L^\infty}+\norm{\nabla \rn}_{L^2}\norm{\de_y\cz}_{L^\infty_y}+\norm{\nabla\rn}_{L^2}\norm{\nabla \cn}_{L^\infty}\\
&\quad+\norm{\rn}_{L^2}\norm{\Delta\cn}_{L^\infty}+\norm{\rz}_{L^\infty}\norm{\rn}_{L^2}+\norm{\rz-\bar\rho}_{L^\infty}\norm{\rn}_{L^2}\Big)\,d\tau\,.
\end{aligned}
\end{equation*}
Applying the Gagliardo-Nirenberg inequalities in below,
\begin{align*}
\norm{\nabla\cn}_{L^\infty}&\leq C\norm{\nabla^2\cn}_{L^4} \leq C\norm{\rn}_{L^4}\leq C\norm{\rn}_{L^2}^{1/4}\norm{\nabla \rn}_{L^2}^{3/4}\,,\\
\norm{\de_y\cz}_{L^\infty_y}&\leq C\norm{\de_{yy}\cz}_{L^1_y}=C\norm{\rz}_{L^1}\leq CM\,,
\end{align*}
we get
\begin{align}
\nonumber
\|\rho_{\neq}(t+s)\|_{L^2}&\leq \|\mathcal{S}_{t}\rho_{\neq}(s)\|_{L^2}+C\nu\int_{s}^{s+t}\Big(\norm{\de_y\rz}_{L^2_y}\norm{\rn}_{L^2}^{1/4}\norm{\nabla\rn}_{L^2}^{3/4}+M\norm{\nabla\rn}_{L^2}\\
\nonumber
&\quad+\norm{ \rn}_{L^2}^{1/4}\norm{\nabla \rn}_{L^2}^{7/4}+\norm{\rn}_{L^2}\norm{\rn}_{L^\infty}+\norm{\rz}_{L^\infty}\norm{\rn}_{L^2}+\norm{\rz-\bar\rho}_{L^\infty}\norm{\rn}_{L^2}\Big)\,d\tau\,.
\end{align}
We then use Corollary~\ref{c:decay}, H\"older's inequality and Assumption~\ref{ass:KSegelBOOT} to get
\begin{align}\label{e:semibound}
\nonumber
\|\rho_{\neq}(t+s)\|_{L^2}&\leq e^{-\lambda_\nu t+\pi/2}\norm{\rn(s)}_{L^2}
+C\Big(\nu\int_s^{s+t}\norm{\de_y\rz}_{L^2_y}^{8/5}\norm{\rn}_{L^2}^{2/5}\,d\tau\Big )^{5/8}\Big(\nu\int_s^{s+t}\norm{\nabla\rn}_{L^2}^2\,d\tau\Big)^{3/8}\\
\nonumber
&\quad+CM(\nu t)^{1/2}\Big(\nu\int_s^{s+t}\norm{\nabla\rn}_{L^2}^2\,d\tau\Big)^{1/2}+C\Big(\nu\int_s^{s+t}\norm{\rn}_{L^2}^2\,d\tau\Big)^{1/8}\Big(\nu\int_s^{s+t}\norm{\nabla\rn}_{L^2}^2\,d\tau\Big)^{7/8}\\
\nonumber
&\quad+C\nu\int_s^{s+t}\norm{\rn}_{L^2}\norm{\rn}_{L^\infty}+\norm{\rz}_{L^\infty}\norm{\rn}_{L^2}+\norm{\rz-\bar\rho}_{L^\infty}\norm{\rn}_{L^2}\,d\tau\\
\nonumber
&\leq 30 e^{-\lambda_\nu t}e^{-\lambda_\nu s/8}\norm{\rho_0}_{L^2}+CC_{\dot H^1}\Big(\frac{\nu}{\lambda_{\nu}}\Big)^{5/8}e^{-\lambda_\nu s/8}\norm{\rho_0}_{L^2}\\
&\quad+CM(\nu t)^{1/2}e^{-\lambda_\nu s/8}\norm{\rho_0}_{L^2}
+C\Big(\frac{\nu}{\lambda_\nu}\Big)^{1/8}e^{-\lambda_\nu s/4}\norm{\rho_0}_{L^2}^2+CC_\infty \frac{\nu}{\lambda_\nu}e^{-\lambda_\nu s/8}\norm{\rho_0}_{L^2}\,.
\end{align}
Taking  $t=\tau^*:=8/\lambda_\nu$, equation~\eqref{e:semibound} yields to
\begin{align}
\|\rho_{\neq}(\tau^*+s)\|_{L^2}&\leq e^{-2}e^{-\lambda_\nu s/8}\norm{\rho_0}_{L^2}+C\Big(\frac{\nu}{\lambda_\nu}\Big)^{1/8}\big(C_{\dot H^1}+M+C_\infty+1\big)e^{-\lambda_\nu s/8}\norm{\rho_0}_{L^2}\,,
\end{align}
where we also used the fact that $\nu /\lambda_\nu\ll 1$ by when $\nu$ is small enough. We further choose $\nu_0$ satisfying
\begin{align}
\frac{\nu_0}{\lambda_{\nu_0}}\leq \Big(\frac{e^{-1}-e^{-2}}{C(C_{\dot H^1}+M+C_\infty+1)}\Big)^{8}\,,
\end{align}
and obtain for any $s\in[0, t_0]$
\begin{align}
\|\rho_{\neq}(\tau^*+s)\|_{L^2}&\leq e^{-1}e^{-\lambda_\nu s/8}\norm{\rho_0}_{L^2}\,.
\end{align}
By taking $s=(n-1)\tau^*$, we get
\begin{align}
\norm{\rn(n\tau^*)}_{L^2}\leq e^{-n}\norm{\rho_0}_{L^2}\,.
\end{align}
Recalling~\eqref{e:rnenergy2}, for any $s\in [0, t_0]$, $0\leq t\leq \tau^*$ with $s+t\leq t_0$,  by choosing
 \begin{align}
 \frac{\nu_0}{\lambda_{\nu_0}}\leq \frac{\ln 2}{8C(C_{\dot H^1}^{8/5}+C_\infty+\norm{\rho_0}_{L^2}^4+\norm{\rho_0}_{L^2}^8)}\,,
 \end{align}
one has
\begin{align*}
\norm{\rn(s+t)}_{L^2}^2\leq 2\norm{\rn(s)}_{L^2}^2\,.
\end{align*}
Since for any $t\in[0, t_0]$, there exists $n$ such that $t\in [(n-1)\tau^*, n\tau^*]$. It then implies
\begin{align}
\norm{\rn(t)}_{L^2}\leq \sqrt 2e^{-(n-1)}\norm{\rho_0}_{L^2}\leq \sqrt 2e^{1-t/\tau^*}\norm{\rho_0}_{L^2}\leq 4 e^{-\lambda_\nu t/8}\norm{\rho_0}_{L^2} \,.
\end{align}
This finishes the proof.
\end{proof}

\section*{Acknowledgments}
The authors thank Anna Mazzucato and Siming He for useful discussions. B. Shi and W. Wang were supported by the National Natural Science Foundation of China ( Grant No.11771284).

\bibliographystyle{plain}
\bibliography{helical}

\end{document}